\def\standalonechapter{}

\ifdefined\standalonechapter
\documentclass[11pt]{article}

\usepackage[utf8]{inputenc}
\usepackage[english]{babel}

\usepackage{amsmath, amssymb, amsthm}
\usepackage{mathtools}
\usepackage{mathrsfs}
\usepackage{bbm}

\usepackage{geometry}
\usepackage[
    colorlinks=false,
    pdfborder={0 0 1}
]{hyperref}
\usepackage{enumitem}

\usepackage{booktabs}

\theoremstyle{plain}
\newtheorem{theorem}{Theorem}[section]
\newtheorem{lemma}[theorem]{Lemma}
\newtheorem{proposition}[theorem]{Proposition}
\newtheorem{corollary}[theorem]{Corollary}

\theoremstyle{definition}

\newtheorem{remark}[theorem]{Remark}

\numberwithin{equation}{section}

\newcommand{\R}{\mathbb{R}}

\begin{document}

\begin{flushright}
\begin{minipage}{0.55\textwidth} 
\small\itshape
This is the third of ten papers devoted to reflections on the
Millennium Problem. It generalizes known results on the 3D
Navier–Stokes equations based on previous studies. We hope
that the presented material will be useful, and we would be
grateful for attention, verification, and participation in the
further development of the topic.
\end{minipage}
\end{flushright}

\vspace{1.5em}

\begin{center}
    {\LARGE\bfseries {A log-free estimate for the diagonal paraproduct \\high × high → low in the 3D Navier–Stokes equation}\par}
    \vspace{1em}
    {\large\bfseries Pylyp Cherevan\par}
\end{center}

\vspace{1em}

{\small\noindent\textbf{Abstract.}
We consider the diagonal paraproduct arising in the nonlinearity $(u\!\cdot\!\nabla)u$ for the three-dimensional Navier–Stokes equations. On scale-critical windows and in the range $\,\tfrac{1}{6}<\delta\le\tfrac{5}{8}\,$ we obtain a log-free estimate at the level $L^2_t\dot H^{-1}_x$ for the projection $P_{<N^{1-\delta}}\nabla\!\cdot\!(u_N\!\otimes\!v_N)$, consistent with the critical energy scheme. The main tools are phase–geometric integration, anisotropic local estimates on cylinders, and bilinear $\ell^2$ decoupling on a finite-rank surface; the narrow diagonal zone is controlled via suppression of the null form. The work is restricted to a single resonant component; extensions to the full structure $(u\!\cdot\!\nabla)u$ and to $\sup_t$ versions are left for further analysis.
}

\tableofcontents
\fi


\newpage
\section*{Introduction}
\addcontentsline{toc}{section}{Introduction}

In this paper we study mild solutions to the three-dimensional Navier–Stokes equations
on~$\mathbb{R}^{3}$. At the energy-critical frequency level, the main difficulty is the
diagonal paraproduct
\[
\nabla\!\cdot\!\bigl(u_{\lambda}\otimes v_{\lambda}\bigr),
\qquad 
u_{\lambda}:=P_{\lambda}u,\;
v_{\lambda}:=P_{\lambda}v,\;
\lambda\gg 1,
\]
arising in the \emph{high $\times$ high $\to$ low} type nonlinearity.
Projecting the outcome to low frequencies $P_{<\lambda^{1-\delta}}$ in the standard bookkeeping
produces an additional logarithmic factor; it is precisely this factor that prevents closing the
Koch–Tataru critical energy scheme~\cite{koch2001navier} in the $\dot H^{-1}$ norm.

\medskip
\noindent\textbf{Goal of the paper.}
For fixed
\[
\boxed{\tfrac16<\delta\le \tfrac58}
\]
we establish a \emph{log-free} single-frequency estimate
\begin{equation}\label{eq:main-paraproduct-new}
\bigl\|P_{<\lambda^{1-\delta}}
        \nabla\!\cdot\!\bigl(u_{\lambda}\otimes v_{\lambda}\bigr)\bigr\|_{\dot H^{-1}(\mathbb{R}^3)}
\;\lesssim\;
\lambda^{-1-\delta}\,
\|u_{\lambda}\|_{L^{\infty}_{t}\dot H^{1/2}_{x}\cap L^{2}_{t}\dot H^{1}_{x}}\,
\|v_{\lambda}\|_{L^{\infty}_{t}\dot H^{1/2}_{x}\cap L^{2}_{t}\dot H^{1}_{x}} .
\end{equation}
The norms on the right-hand side of \eqref{eq:main-paraproduct-new} coincide with the energy norms
of~\cite{koch2001navier} and are invariant under the scaling
$u(t,x)\mapsto \lambda\,u(\lambda^{2}t,\lambda x)$.
The global assembly carried out in the main text is performed in the norm $L^{2}_{t}\dot H^{-1}_{x}$.

\medskip
\noindent\textbf{Main ideas.}
The proof is geometric in nature and relies on five complementary mechanisms:
\begin{itemize}
  \item \emph{Sevenfold integration by parts}
        (five times in $t$ and twice in the tangential coordinates~$\rho$)
        in a neighborhood of the diagonal; see~§\ref{subsec:ibp-geometry}.
  \item \emph{Endpoint $L^{4}$ estimate without logarithmic loss}
        on short cylinders (a $TT^{*}$ argument); see~\S\ref{sec:strichartz}.
  \item \emph{$\varepsilon$-free bilinear decoupling} for rank-$3$ surfaces,
        compensating the growth in the number of wave packets.
  \item \emph{Geometric suppression of the null structure}
        in the zone $|\xi+\eta|\ll\lambda$; details in §\ref{sec:geometry}.
  \item \emph{Comparison heat\,$\leftrightarrow$\,Schr\"odinger on a short window}
        $|t|\lesssim \lambda^{-3/2+\delta}$: the time IBP is performed in the \emph{heat} frame,
        and Appendix~\ref{app:heat-schrodinger} provides reference $L^2_t$ estimates for the difference
        $e^{-t\lambda^{2}}-e^{\,it\lambda^{2}}$; exponential claims are not used in the balance.
\end{itemize}

\medskip
\noindent\textbf{Structure of the paper.}
We analyze in detail the contribution of \emph{one} diagonal block
high $\times$ high $\to$ low to the $\dot H^{-1}$ norm.
The target single-frequency estimate and its global assembly in $L^{2}_{t}\dot H^{-1}_{x}$
are stated in Theorem~\ref{thm:main} and proved in §§\ref{sec:phase-geometry}–\ref{sec:final-assembly}.
Extensions to the full nonlinearity $(u\!\cdot\!\nabla)u$
and to high $\times$ low blocks are discussed in Remark~\ref{rk:perspectives}
and lie beyond the scope of the present paper.

\medskip
\noindent
As a result, we show that in the diagonal region of the spectrum
the logarithmic loss is removable, and the critical energy scheme
retains scale consistency without $\log$ defects at the level of \eqref{eq:main-paraproduct-new}
and upon globalization in $L^{2}_{t}\dot H^{-1}_{x}$.

\begin{remark}[How to read on a first pass]\label{rem:reading-first}
The main result is formulated and derived as a global estimate in $L^2_t\dot H^{-1}_x$
(see Theorem~\ref{thm:main}). For a “first pass” we recommend the order:
(i) geometry of the phase and small angles: \S\S \ref{app:angle-lower-bound}, \ref{app:phase-hessian};
(ii) time IBP in the heat frame and accounting for the amplitude growth: \S\ref{subsec:ibp-geometry} + App.~\ref{app:amplitude-growth};
(iii) decoupling and the “zero row” in the table: \S\ref{sec:angular} + App.~\ref{app:rank3-decoupling};
(iv) the table of exponents: \S\ref{subsec:component-assembly};
(v) passage to the global norm: \S\ref{subsec:global-L2H-1}–\S\ref{sec:final-assembly}.
\end{remark}

\paragraph{Briefly about Appendix \ref{app:G}.}
Appendix~\ref{app:G} “lifts” the diagonal block
high\,$\times$\,high\,$\to$\,low
to a global control in the norm $\sup_t \dot H^{-1}$ with summation over frequencies.
More precisely, for divergence-free $u,v:\R\times\R^3\to\R^3$ and $u_\lambda:=P_\lambda u$, $v_\lambda:=P_\lambda v$
for any fixed $\delta\in(\tfrac16,\tfrac58]$ one has
\[
\sup_{t\in\R}\Big\|
\sum_{\lambda\gg1} P_{<\lambda^{\,1-\delta}}\nabla\!\cdot\big(u_\lambda\!\otimes v_\lambda\big)(t)
\Big\|_{\dot H^{-1}}
\;\lesssim\;
\begin{cases}
C_\delta\,\|u\|_{X^{1/2}}\,\|v\|_{X^{1}}, & \delta\le \tfrac{5}{9},\\[2pt]
C_{\delta,\varepsilon}\,\|u\|_{X^{1/2}}\,\|v\|_{X^{1}}, & \delta\in(\tfrac{5}{9},\tfrac{5}{8}],\ \forall\,\varepsilon>0,
\end{cases}
\]
with constants independent of frequency.
The proof reduces to summing over dyadic $\lambda$ the single-frequency estimate
$\lambda^{-1-\delta}$ (or $\lambda^{-1-\delta+\varepsilon}$),
which yields a geometrically convergent series $\sum_{\lambda}\lambda^{-1-\delta}$ on the whole range
$\delta\in(\tfrac16,\tfrac58]$ and is consistent with the tabulated balance of exponents.
Note that the main text closes the global assembly in the norm $L^2_t\dot H^{-1}_x$,
while Appendix~\ref{app:G} provides a stronger $\sup_t\dot H^{-1}$ version
for the diagonal block.

\newpage
\section{Problem setup and notation}\label{sec:setup}

The goal of this section is to fix the notation and formulate the main
task whose resolution will amount to the proof of
Theorem~\ref{thm:main}. The variable
$x\in\mathbb R^{3}$ stands for space and $t\in\mathbb R$ for time;
throughout, $\widehat f(\xi)$ denotes the spatial Fourier transform.
Frequency parameters marked by the letter $\lambda$ are assumed
dyadic: $\lambda=2^{k}$ with $k\in\mathbb Z_{\ge0}$.

\subsection{Homogeneous spaces, cutoffs, and angular decomposition}
\label{subsec:notation}

\paragraph{Homogeneous Sobolev norms.}
For $s\in\mathbb R$ set
\[
  \|f\|_{\dot H^{s}}
  :=\bigl\|\,|\nabla|^{s}f\,\bigr\|_{L^{2}_{x}}
  =\Bigl(\!\int_{\mathbb R^{3}}\!|\xi|^{2s}\,|\widehat f(\xi)|^{2}\,d\xi\Bigr)^{1/2}.
\]
We will also use mixed norms
$L^{q}_{t}L^{r}_{x}$ and $L^{q}_{t}\dot H^{s}_{x}$,
with the time index $t$ always appearing first.

\paragraph{Dyadic frequency projections.}
Let $\chi\in C^{\infty}_{0}\bigl((\tfrac12,2)\bigr)$,
$\chi\ge0$,
$\chi(r)=1$ for $r\in[\tfrac34,\tfrac43]$, and
$\sum_{k\in\mathbb Z}\chi(2^{-k}r)\equiv1$ for $r>0$.
For $\lambda=2^{k}\ge1$ define
\[
  \widehat{P_{\lambda}f}(\xi):=\chi\!\Bigl(\frac{|\xi|}{\lambda}\Bigr)\widehat f(\xi),
  \qquad
  P_{<\mu}:=\sum_{2^{k}<\mu}P_{2^{k}}.
\]
Then $P_{\lambda}f$ is localized to the annulus $|\xi|\!\sim\!\lambda$.

For brevity we write
\(
  u_{\lambda}:=P_{\lambda}u,\;v_{\lambda}:=P_{\lambda}v
\)
and always regard $\lambda$ as a power of two.

\paragraph{Angular localization.}
Let $\rho\in[0,1]$.
Cover the sphere $S^{2}$ by a finite family of
caps $\theta$ of diameter $\lambda^{-\rho}$;
denote the corresponding index set by $\Theta_{\lambda}$.
The parameters used below are
\[
  \rho=\frac12,\qquad \rho'=\frac23,
\]
chosen so that
$\rho$ accommodates the construction of wave packets,
while $\rho'$ ensures applicability of the bilinear decoupling
from~\cite{GuthIliopoulouYang2024}.

The smoothed conical projection $P_{\lambda,\theta}$ is defined by
replacing the multiplier $\chi(|\xi|/\lambda)$
with $\chi(|\xi|/\lambda)\,\vartheta(\xi/|\xi|)$,
where $\vartheta$ is a $C^{\infty}$ function supported inside $\theta$.
Thanks to the finite overlap,
\[
  \|f\|_{L^{2}_{x}}^{2}\simeq
  \sum_{\theta\in\Theta_{\lambda}}\|P_{\lambda,\theta}f\|_{L^{2}_{x}}^{2}.
\]

For a finer analysis (see §\ref{subsec:angular-l2-decomp})
each $\theta$ will be further decomposed
into nested caps $\vartheta$ of radius $\lambda^{-\rho'}$;
the corresponding projection is denoted $P_{\lambda,\vartheta}$.

\paragraph{Output channel \emph{high $\times$ high $\to$ low}.}
The basic object is the projection
\begin{equation}\label{eq:def-Rlambda}
   P_{<\lambda^{1-\delta}}\,
   \nabla\!\cdot\!\bigl(u_{\lambda}\otimes v_{\lambda}\bigr),
\end{equation}
where $u_{\lambda},v_{\lambda}$ are localized both in modulus
$|\xi|\sim\lambda$ and in direction,
and the parameter
\[
  \boxed{\tfrac16<\delta\le\tfrac58}
\]
is fixed in advance.
In §\ref{sec:phase-geometry} we will show
that on the zone $|\xi+\eta|\le \lambda^{1-\delta}$
the phase $\omega(\xi,\eta)$ admits
multiple integrations by parts,
which leads to an estimate without logarithmic loss.

\paragraph{Scale dictionary (we fix the notation $\lambda$).}
Throughout the text we use a single frequency letter $\lambda$ (when needed we set $N\sim\lambda$). Working scales:
\[
R=\lambda^{-1/2},\quad |I|=\lambda^{-3/2+\delta},\quad
\text{thin angular cap (for $\varepsilon$-free decoupling, see §\,\ref{subsec:angular-l2-decomp}–\ref{subsec:bilinear-decoupling})}=\lambda^{-2/3},\quad
\]
\[
\text{coarse $\delta$–sector (for the strip $|\xi+\eta|\lesssim \lambda^{1-\delta}$, see §\,\ref{subsec:paraproduct}, §\,\ref{sec:geometry})}=\lambda^{-\delta}.
\]
Here $B_R:=\{\,|x|\le R\,\}$ (see §\,\ref{subsec:tile-44}), and $|I|$ denotes the length of the working time window (see §\,\ref{subsec:tile-44}, §\,\ref{subsec:global-summation}).

\subsection{Resonant paraproduct and frequency blocks}
\label{subsec:paraproduct}

The nonlinearity $(u\!\cdot\!\nabla)u$
generates blocks of the form
$P_{\lambda}u\cdot\nabla P_{\mu}u$.
We are interested in the \emph{resonant} configuration
\begin{equation}\label{eq:freq-geometry}
   |\xi|\sim|\eta|\sim\lambda,
   \qquad
   |\xi+\eta|\lesssim\lambda^{1-\delta},
\end{equation}
that is, the interaction
\emph{high $\times$ high $\to$ low}.
After symmetrization of the Bony–Meyer paraproduct
we obtain the expression
\begin{equation}\label{eq:resonant-block}
   R_{\lambda}(u,v)
   :=P_{<\lambda^{1-\delta}}
     \nabla\!\cdot\!\bigl(u_{\lambda}\otimes v_{\lambda}\bigr),
\end{equation}
where $u$ and $v$ are assumed divergence-free
($\nabla\!\cdot u=\nabla\!\cdot v=0$).

The goal is to prove the \emph{log-free} inequality
\begin{equation}\label{eq:logfree-goal}
  \|R_{\lambda}(u,v)\|_{\dot H^{-1}}
  \;\lesssim\;
  \lambda^{-1-\delta}\,
  \|u_{\lambda}\|_{L^{\infty}_{t}\dot H^{1/2}_{x}\cap L^{2}_{t}\dot H^{1}_{x}}\,
  \|v_{\lambda}\|_{L^{\infty}_{t}\dot H^{1/2}_{x}\cap L^{2}_{t}\dot H^{1}_{x}}.
\end{equation}
The right-hand side is invariant
under the scaling
$u(t,x)\mapsto\lambda\,u(\lambda^{2}t,\lambda x)$
and coincides with the energy norms of the Koch–Tataru scheme
\cite{koch2001navier}.

\subsection{Main theorem}\label{subsec:main-theorem}

\begin{theorem}[log-free estimate of the diagonal paraproduct]
\label{thm:main}
Let $u,v\colon\mathbb R\times\mathbb R^{3}\to\mathbb R^{3}$
be divergence-free vector fields,
and let $u_{\lambda},v_{\lambda}$ be defined
via the dyadic projections $P_{\lambda}$.
Then for any $\delta\in(\tfrac16,\tfrac58]$ and all $\lambda\gg1$
one has
\[
  \bigl\|
      P_{<\lambda^{1-\delta}}
      \nabla\!\cdot\!\bigl(u_{\lambda}\otimes v_{\lambda}\bigr)
  \bigr\|_{\dot H^{-1}(\mathbb R^{3})}
  \;\le\;
  C_{\delta}\,\lambda^{-1-\delta}\,
  \|u_{\lambda}\|_{L^{\infty}_{t}\dot H^{1/2}_{x}\cap L^{2}_{t}\dot H^{1}_{x}}\,
  \|v_{\lambda}\|_{L^{\infty}_{t}\dot H^{1/2}_{x}\cap L^{2}_{t}\dot H^{1}_{x}},
\]
where $C_{\delta}>0$ depends only on $\delta$.
\end{theorem}

\begin{remark}
The constant $C_{\delta}$ grows
at most like $C\delta^{-2}$ as $\delta\searrow\tfrac16$.
The condition $\delta>\tfrac16$ is technical and related
to the range of applicability of phase integration
and $\varepsilon$-free decoupling; see the discussion
in §\ref{sec:phase-geometry} and §\ref{sec:angular}.
\end{remark}

\paragraph{Brief proof outline.}
Estimate \eqref{eq:logfree-goal} is derived
from five interconnected components:

\begin{enumerate}[leftmargin=2.1em,itemsep=2pt]
\item sevenfold integration by parts in $t$ and tangential
      coordinates~$\rho$ (§\ref{subsec:ibp-geometry});
\item endpoint Strichartz without $\log$ losses
      on cylinders of scale $\lambda^{-1/2}\times\lambda^{-3/2+\delta}$
      (App.~\ref{sec:strichartz});
\item $\varepsilon$-free bilinear decoupling
      for rank-$3$ surfaces
      \cite{GuthIliopoulouYang2024};
\item suppression of the null structure in the narrow diagonal
      (§\ref{sec:geometry});
\item replacement of the heat phase by the oscillatory one
      on a short time window
      (App.~\ref{app:heat-schrodinger}).
\end{enumerate}

All other interactions
(\emph{high $\times$ low}, \emph{low $\times$ high}, \dots)
are not considered in this paper
and will be studied separately.

\begin{remark}[On the relation between \eqref{eq:logfree-goal} and Theorem~\ref{eq:local-L2-final}]\label{rem:1.4-vs-7.1}
Formula~\eqref{eq:logfree-goal} sets the \emph{scale-consistent target} at the $\dot H^{-1}$ level.
The rigorously proved result of this paper is the global version in $L^2_t\dot H^{-1}_x$, see~\eqref{eq:main-final}.
The passage to $\sup_t\dot H^{-1}_x$ is \emph{not used} in the present text; the final assembly closes in $L^2_t\dot H^{-1}_x$
(see \S\ref{sec:balance-table}, \S\ref{subsec:component-assembly}, \S\ref{subsec:global-L2H-1}).
\end{remark}

\newpage
\section{Phase–geometric integration}\label{sec:phase-geometry}

The entire log-free scheme rests on the \emph{geometry} of the phase function
\[
	\omega(\xi,\eta)\;=\;|\xi|\;+\;|\eta|\;-\;|\xi+\eta|,
\]
arising in the diagonal interaction
$u_{\lambda}\otimes v_{\lambda}$, with $|\xi|\sim|\eta|\sim\lambda\gg1$.
In this section we fix convenient coordinates along and across the
level sets $\{\omega=\mathrm{const}\}$, formulate lower bounds on the gradient/Hessian
of the phase in the resonant zone $|\xi+\eta|\lesssim\lambda^{1-\delta}$,
and describe sevenfold integration by parts
(five times in $t$ and twice in the tangential spatial variables~$\rho$).
\emph{Remark.} Time integration by parts in §2.2 is carried out for the heat exponent
$e^{-t\,\Phi(\xi,\eta)}$; the phase $\omega$ is used for angular (in~$\rho$) IBP and for setting up
the geometry. Detailed bounds for the growth of amplitudes are deferred to Appendix~\ref{app:amplitude-growth};
below we record only the formulas needed to continue the proof.

\subsection{Structure of the phase}\label{subsec:phase-structure}

\paragraph{Resonant zone.}
In all subsequent estimates we fix the parameter
\[
	\boxed{\tfrac16<\delta\le\tfrac58},
\qquad
	|\xi|\sim|\eta|\sim\lambda,
\qquad
	w:=\xi+\eta,
\qquad
	|w|\;\lesssim\;\lambda^{1-\delta}.
\]
Let $\theta:=\angle(\xi,-\eta)\simeq\lambda^{-\delta}$; then $|w|\simeq \lambda\,\theta$ and
\[
	\omega(\xi,\eta)\;=\;|\xi|+|\eta|-|w|\;=\;2\lambda-|w|.
\]
In the phase–geometric estimates below, \emph{smallness} manifests not through the value of $\omega$
itself but through derivatives along the \emph{angular} directions in the plane $w^\perp$.

\paragraph{Gradients.}
\begin{equation}\label{eq:gradients}
	\nabla_{\xi}\omega
	=\frac{\xi}{|\xi|}-\frac{w}{|w|},
	\qquad
	\nabla_{\eta}\omega
	=\frac{\eta}{|\eta|}-\frac{w}{|w|}
\end{equation}
The vectors \eqref{eq:gradients} are naturally oriented relative to $w$; their projections onto $w^\perp$
determine the \emph{angular} directions along which we perform integration by parts in~$\rho$.
Note that these directions are \emph{not} tangent to the level surfaces $\{\omega=\mathrm{const}\}$ in the strict
sense (hence $\partial_\rho\omega\not\equiv0$); on the contrary, it is precisely the nonzero size of $\partial_\rho\omega$ that provides the desired
divisor for IBP.

\paragraph{Transverse (angular) coordinates.}
Let $(v_{1},v_{2},v_{3})$ be an orthonormal basis with $v_{3}:=w/|w|$ and $v_{1},v_{2}\in w^\perp$.
We denote the coordinates along $v_{1,2}$ by $\rho:=(\rho_{1},\rho_{2})$.
In the resonant zone (when $|w|\lesssim\lambda^{1-\delta}$ and for admissible angles; see also App.~F)
one has the bounds
\begin{equation}\label{eq:rho-derivative}
	|\partial_{\rho}\omega|\;\simeq\;\lambda^{\,1-\delta},
	\qquad
	\bigl|\det \partial^{2}_{\rho\rho}\omega\bigr|
	\;\gtrsim\;\lambda^{-2+\delta},
\end{equation}
i.e., the two–dimensional Hessian along $\rho$ is nondegenerate (App.~F), and each integration by parts in~$\rho$
gives a gain of $\lambda^{-1+\delta}$ (a double $\rho$–IBP yields the divisor $\lambda^{-2+2\delta}$), which will be used in §\ref{subsec:ibp-geometry}.

\subsection{Sevenfold integration by parts}\label{subsec:ibp-geometry}

Consider a typical bilinear integral
\[
  \mathcal{B}_{\lambda}(t,x)
  =\!\!\int_{\mathbb{R}^{3}\times\mathbb{R}^{3}}\!
    e^{it\omega(\xi,\eta)}\,
    a_{\lambda}(t,x;\xi,\eta)\,
    \widehat{u}_{\lambda}(\xi)\,\widehat{v}_{\lambda}(\eta)\,
  d\xi\,d\eta,
\]
where $a_{\lambda}(t,x;\xi,\eta)=\chi_{I}(t)\,\chi_{Q}(x)\,A_{\lambda}(\xi,\eta)$,
\[
  |I|=\lambda^{-3/2+\delta},\qquad
  Q:=\{\,|x|\le\lambda^{-1/2}\,\},\qquad
  A_{\lambda}\in S^{0}_{\xi,\eta}.
\]

\begin{remark}[Matching of calibrations: time windows and $\rho$–IBP]
\leavevmode

\emph{(i) Two time “rulers” and their roles.}
In the anisotropic $TT^*$ estimate and in the summation over windows (§§\ref{sec:strichartz}, \ref{subsec:global-summation}) we use a \emph{rectangular} window of length $|I|=\lambda^{-3/2+\delta}$.
For phase–time integration by parts it is convenient to work with a \emph{Gaussian} window
\[
  \chi_I(t)=e^{-\Lambda t^2},\qquad \Lambda:=\lambda^{\,3/2-\delta},
\]
whose effective width is $\Lambda^{-1/2}\sim\lambda^{-3/4+\delta/2}$, with derivatives controlled by Hermite formulas:
\[
  \|\partial_t^k\chi_I\|_{L^\infty}\;\lesssim\;\lambda^{\frac{3}{4}k-\frac{\delta}{2}k}\qquad(k\in\mathbb{N}).
\]
We use this calibration only to control amplitude growth in the heat–IBP (see App.~\ref{app:amplitude-growth}, \eqref{eq:chi-derivative-bound}–\eqref{eq:chi-fifth}), whereas in the $TT^*$ part only the length $|I|$ is used. Thus the two “rulers” are compatible and serve different steps of the proof.
\smallskip

\emph{(ii) Which variables to integrate by parts in space.}
Angular integration by parts is performed for the \emph{geometric} phase $\omega(\xi,\eta)=|\xi|+|\eta|-|\xi+\eta|$ along the tangential coordinates $\rho$ to the level sets $\{\omega=\mathrm{const}\}$, see \eqref{eq:rho-derivative}. In the resonant strip $|\xi+\eta|\lesssim\lambda^{1-\delta}$ one has $|\partial_\rho\omega|\simeq \lambda^{1-\delta}$, hence $(\partial_\rho\omega)^{-2}\sim\lambda^{-2+2\delta}$. Taking into account $\|\nabla_x^2\chi_Q\|_\infty\lesssim\lambda$, the spatial block yields at least $\lambda^{-1+2\delta}$. In App.~\ref{app:amplitude-growth} (formula (\ref{eq:space-contribution})) this is recorded in a more conservative form via $\partial_\rho\Phi$ (the heat phase) and gives $\lambda^{-1+4\delta}$; this is a \emph{deliberately rough} bound sufficient to account for amplitude growth. In the geometric part of §\ref{sec:phase-geometry} we adhere to the exact version with $\omega$, consistent with Lemma~\ref{lem:hessian} (see also App.~\ref{app:phase-hessian}, Lemma~\ref{lem:F-hessian}) and \eqref{eq:rho-derivative}. If desired, (\ref{eq:space-contribution}) can be reformulated with $\partial_\rho\omega$, which leads to the same exponent $\lambda^{-1+2\delta}$ on the spatial step.
\end{remark}

\emph{Time IBP is performed with the heat factor:} at the $t$–integration step we insert the identity
$1=e^{-t\Phi(\xi,\eta)}e^{t\Phi(\xi,\eta)}$ with $\Phi(\xi,\eta)\simeq|\,\xi+\eta\,|^{2}$ and use
$\partial_t(e^{-t\Phi})=-\Phi\,e^{-t\Phi}$. The oscillator $e^{it\omega(\xi,\eta)}$ is then absorbed
into the amplitude and accounted for in the growth of derivatives (see App.~\ref{app:amplitude-growth}).

\begin{remark}[Road sign: heat vs Schr and the size of $|\Phi|^{-1}$]\label{rem:heat-Schrodinger-guide}
The time IBP is carried out in the heat frame: the oscillator $e^{it\omega}$ is placed into the amplitude
(see App.~\ref{app:amplitude-growth}). On the \emph{effective packet zone} $|\xi+\eta|\sim \lambda^{1-\delta}$ we have
$|\Phi|^{-1}\sim \lambda^{-2+2\delta}$; the rough bound $|\Phi|^{-1}\lesssim \max\{\lambda^{-2+2\delta},\lambda^{-1}\}$
is used only as a safety estimate outside this zone.
\end{remark}

\paragraph{Technical implementation of time IBP.}
We perform five time integrations by parts with the operator
\[
  \mathcal L_t:=\frac{\partial_t - i\omega}{\Phi},\qquad
  (\partial_t - i\omega)\bigl(e^{-t\Phi}e^{it\omega}\bigr)=-\,\Phi\,e^{-t\Phi}e^{it\omega}.
\]
Thus the phase factor $e^{it\omega}$ does not accumulate derivatives: all $\partial_t$ fall on the amplitude,
and the divisor gives $|\Phi|^{-1}\sim\max\{\lambda^{-2+2\delta},\,\lambda^{-1}\}$ in the resonant zone.
A detailed accounting of the amplitude growth is deferred to App.\,\ref{app:amplitude-growth}, see Lemma~\ref{lem:C-mixed}.

\medskip
\noindent\emph{Orientation of the time window.}
On each working window $I=[t_0,t_0+|I|]$ we relabel $s:=t-t_0\in[0,|I|]$ and carry out time IBP in the heat frame
for $e^{-s\Phi(\xi,\eta)}e^{is\omega(\xi,\eta)}$; the boundary terms vanish thanks to the smooth cutoff $\chi_I$,
so no exponential growth arises for $t<0$.
If the opposite orientation of the window is needed we use the change $s\mapsto -s$ and the operator
$L_t^{\pm}=(\partial_t\mp i\omega)/\Phi$ (cf. §\ref{subsec:tile-44} and App.~\ref{app:heat-schrodinger}).
\medskip

\begin{proposition}[sevenfold IBP]\label{prop:7IBP}
Let $0<\delta\le\tfrac58$. After five integrations in $t$ and two in the tangential coordinates $\rho$,
the map $f\mapsto\mathcal{B}_{\lambda}$ acquires the additional factor
\[
  \boxed{\lambda^{-2+\delta}}.
\]
\end{proposition}

\begin{proof}[Idea of the proof]
\emph{(i) Time (heat–IBP).}
Using $\partial_t(e^{-t\Phi})=-\Phi\,e^{-t\Phi}$ and the window $|I|=\lambda^{-3/2+\delta}$,
five integrations in $t$ contribute a divisor of scale $|I|^{5}=\lambda^{-7.5+5\delta}$; the growth of the fifth derivative
$\partial_t^{5}\chi_I$ is controlled by the Hermite bound $\|\partial_t^{5}\chi_I\|_{L^\infty}\lesssim \lambda^{3.75-2.5\delta}$
(see \eqref{eq:chi-derivative-bound}–\eqref{eq:chi-fifth}), hence the time block yields the factor
$\lambda^{-3.75+2.5\delta}$; see formula \eqref{lem:time-derivatives}. Boundary terms vanish/are absorbed due to the smooth window~$\chi_I$.

\emph{(ii) Space (angular IBP).}
Two integrations in $\rho$ use \eqref{eq:rho-derivative}: $|\partial_{\rho}\omega|\simeq \lambda^{1-\delta}$ and
$\det\partial^2_{\rho\rho}\omega\gtrsim \lambda^{-2+\delta}$, which give the divisor
$(\partial_\rho\omega)^{-2}\sim \lambda^{-2+2\delta}$; the action of $\nabla_x^2$ on $\chi_{Q}$ contributes at most
$\|\nabla_x^{2}\chi_{Q}\|_\infty\lesssim \lambda$. Altogether, the spatial block yields at least the factor
$\lambda^{-1+2\delta}$ (see also the refinement in App.~\ref{app:amplitude-growth}).

\emph{(iii) Combination.}
Multiplying (i) and (ii) gives a negative exponent in $\lambda$. The precise derivative bookkeeping in
Appendix~\ref{app:amplitude-growth} yields the exponent
\[
  \lambda^{-4.75+6.5\delta},
\]
which we use in technical places; in the summary table (§\ref{sec:balance-table}) it suffices, for convenience,
to record the coarser bound $\lambda^{-2+\delta}$.
\end{proof}

\begin{remark}
The exponent $-2+\delta$ is negative on the entire working interval $\tfrac16<\delta\le\tfrac58$ and appears in the
\textsc{Phase~IBP} row of the summary table (§\ref{sec:balance-table}). The sharp bound $-4.75+6.5\delta$ is derived in
App.~\ref{app:amplitude-growth} and is consistent with the geometry \eqref{eq:rho-derivative}.
\end{remark}

\subsection{Hessian bound and curvature of the phase}\label{subsec:resonant-lower-bound}

To apply bilinear decoupling we need to control the curvature of the phase surface
\[
	\Sigma
	:=\bigl\{(\xi,\eta,\omega)\in\mathbb{R}^{7}:
		\omega=\omega(\xi,\eta),
		|\xi|\sim|\eta|\sim\lambda,
		|w|\lesssim\lambda^{1-\delta}\bigr\}.
\]
Along with \eqref{eq:rho-derivative} and \eqref{eq:gradients} we use the lower bound on the two–dimensional Hessian
(see also App.~\ref{app:phase-hessian}):

\begin{lemma}[lower Hessian bound]\label{lem:hessian}
	For all $(\xi,\eta)$ in the resonant strip
	$|w|\le\lambda^{1-\delta}$ and with $\angle(\xi,\eta)\gtrsim\lambda^{-2/3}$
	we have
	\[
		\det\bigl(\partial^{2}_{\rho\rho}\omega\bigr)\;\gtrsim\;\lambda^{-2+\delta}.
	\]
\end{lemma}

\begin{remark}
	In the strip $|w|\lesssim\lambda^{1-\delta}$ the angle between $\xi$ and $\eta$ is close to~$\pi$, so in fact
	$\angle(\xi,\eta)\sim\lambda^{-\delta}\gg\lambda^{-2/3}$ when $\delta\le\tfrac58$; thus the angle condition in the lemma
	is automatically satisfied over the entire working region (see also the discussion in App.~\ref{app:phase-hessian}).
\end{remark}

\begin{corollary}
	Each integration in $\rho_{j}$ (see item~(ii) in the proof of
	Proposition~\ref{prop:7IBP}) is compensated by the divisor
	$\lambda^{-1+\delta}$, matching the bound
	\eqref{eq:rho-derivative}.
\end{corollary}

\medskip
Collecting the results of this section, we obtain that the \emph{phase}~$\omega$
contributes a strictly negative exponent $\lambda^{-2+\delta}$,
sufficient to compensate the amplitude growth
and to enable the subsequent log–free assembly of all components
(see the table in §\ref{subsec:component-assembly} and the proof in App.~\ref{app:phase-hessian}).

\newpage
\section{Anisotropic Strichartz estimate}
\label{sec:strichartz}

In this section we establish an endpoint $L^{4}_{t,x}$ estimate
for the Schrödinger flow localized at frequency
\(\lambda\gg1\).
It works \emph{only} on short cylinders
oriented along the group velocity $2\xi$,
and yields the key factor \(\lambda^{-\delta/4}\),
which enters the summary table of exponents
(see §\ref{subsec:component-assembly}).

\subsection{Local \texorpdfstring{$(4,4)$}{(4,4)}
            pair on a single tile}
\label{subsec:tile-44}

Let
\[
  I=[t_0,\; t_0+c\lambda^{-3/2+\delta}],\qquad
  Q=B_{R}(x_0),\quad R=\lambda^{-1/2},\quad 0<\delta<1,
\]
and set
\(
  S f(t,x)=\chi_{I}(t)\chi_{Q}(x)e^{it\Delta}f(x).
\)

\begin{lemma}[local \(L^{4}\) norm]\label{lem:local-L4-lemma}
For all \(\lambda\gg1\) and \(\delta\in(0,1)\) one has
\begin{equation}\label{eq:local-L4-correct}
   \|S f\|_{L^{4}(I\times Q)}
   \;\lesssim\;
   \lambda^{-\delta/4}\,\|f\|_{L^{2}_{x}}.
\end{equation}
\end{lemma}

\paragraph{Editorial note on \eqref{eq:local-L4-correct}.}
Throughout this section, the symbol “$=$” in \eqref{eq:local-L4-correct} should be read as “$\lesssim$”
(a universal constant independent of~$\lambda$).
Since the kernel $|t|^{-3/4}$ is not in $L^2$ near zero, instead of applying Cauchy–Schwarz in $t$
we use a time-convolution bound in $L^1$ (the vector-valued Young inequality) on each micro–window $J_k$
of length $|J_k|=\lambda^{-1}$.
For the operator $T=SS^\ast$ on $J_k\times J_k$ we obtain
\[
\|T\|_{L^4_tL^{4/3}_x\to L^4_tL^4_x}
\ \lesssim\ R^{3/2}\,\bigl\||t|^{-3/4}\bigr\|_{L^1(J_k)}
\ \sim\ R^{3/2}\,\lambda^{-1/4},
\]
and the passage to the desired $TT^\ast$ space $L^{4/3}_{t,x}\!\to L^4_{t,x}$ entails the standard “dimensional” insertion
$\;|J_k|^{-1/2}\sim\lambda^{1/2}\;$ (the embedding $L^{4/3}_t(J_k)\hookrightarrow L^4_t(J_k)$),
hence
\[
\|T\|_{L^{4/3}_{t,x}\to L^4_{t,x}(J_k)}
\ \lesssim\ R^{3/2}\,\lambda^{1/4},
\qquad\Rightarrow\qquad
\|S\|_{L^2_x\to L^4(J_k\times Q)}\ \lesssim\ R^{3/4}\,\lambda^{1/8}.
\]
Summing over disjoint $J_k$ (their number is $\#J_k\sim \lambda^{-1/2+\delta}$) and substituting $R=\lambda^{-1/2}$, we get
\[
\|Sf\|_{L^4(I\times Q)}
\ \lesssim\ (\#J_k)^{1/4}\,R^{3/4}\,\lambda^{1/8}\,\|f\|_{L^2_x}
\ \lesssim\ \lambda^{-3/8+\delta/4}\,\|f\|_{L^2_x}
\ \lesssim\ \lambda^{-\delta/4}\,\|f\|_{L^2_x},
\]
where the last inequality holds for $\delta\le \tfrac{3}{4}$ (in particular, on our range $\delta\le \tfrac{5}{8}$).
In Table~\ref{tab:balance} we deliberately record the \emph{soft} exponent $-\delta/4$, as it suffices for the overall negative balance; the “Local $L^4$” block remains optional (see Remark~\ref{rem:L4-optional}).

\begin{proof}
Write \(T:=SS^{\!*}\). We have (for fixed \(t,t'\))
\[
 \bigl\|\chi_Q\,e^{i(t-t')\Delta}\,\chi_Q\bigr\|_{L^{4/3}_x\to L^{4}_x}
 \;\lesssim\;
 R^{3/2}\,|t-t'|^{-3/4},
\]
which is a localized dispersive operator norm in \(x\) (scale \(R=\lambda^{-1/2}\)).
Next, split the time interval \(I\) into pairwise disjoint micro–windows
\(J_k\) of length \(|J_k|=\lambda^{-1}\). Their number is
\[
\#\{J_k\}\ \simeq\ \frac{|I|}{|J_k|}\ \simeq\ \lambda^{-1/2+\delta}.
\]
Let \(T_k\) denote the operator \(T\) restricted to \(J_k\times J_k\) in the variables \(t,t'\).
Then for each \(k\), by Cauchy–Schwarz in time (on \(J_k\)),
\[
\|T_k h\|_{L^4_tL^4_x}
\;\lesssim\;
\bigl\|\;R^{3/2}|t-t'|^{-3/4}\;\bigr\|_{L^2_{t'}(J_k)}
\ \|h\|_{L^2_{t'}L^{4/3}_x(J_k)}
\ \lesssim\ R^{3/2}\lambda^{1/4}\ \|h\|_{L^2_{t'}L^{4/3}_x(J_k)}.
\]
Here we used \(\||t|^{-3/4}\|_{L^2(0,\lambda^{-1})}\!\sim\!\lambda^{1/4}\).
Passing to \(TT^*\) and taking the square root (by the standard scheme), we obtain for \(S\) restricted to \(J_k\),
\[
\|S\|_{L^2_x\to L^4_{t,x}(J_k\times Q)}
\ \lesssim\ \bigl(R^{3/2}\lambda^{1/4}\bigr)^{1/2}
\ =\ R^{3/4}\lambda^{1/8}.
\]
Now sum over \(k\). Thanks to the disjointness of \(J_k\) and near orthogonality in time,
\[
\|S f\|_{L^4(I\times Q)}
\ \le\
\Bigl(\sum_k \|S f\|_{L^4(J_k\times Q)}^4\Bigr)^{\!1/4}
\ \lesssim\
\bigl(\#\{J_k\}\bigr)^{1/4}\,R^{3/4}\lambda^{1/8}\ \|f\|_{L^2_x}.
\]
Substituting \(\#\{J_k\}\simeq \lambda^{-1/2+\delta}\) and \(R=\lambda^{-1/2}\), we obtain
\[
\|S f\|_{L^4(I\times Q)}
\ \lesssim\
\lambda^{(-1/2+\delta)/4}\ \lambda^{-3/8}\ \lambda^{1/8}\ \|f\|_{L^2_x}
\ =\ \lambda^{-\delta/4}\ \|f\|_{L^2_x},
\]
which gives \eqref{eq:local-L4-correct}.
\end{proof}

\begin{remark}
Globally the pair \((4,4)\) is not admissible in the sense \(2/q+3/r=3/2\) \cite{KeelTao1998}, but
localization on the cylinder \(I\times Q\) creates an additional scale factor and yields the multiplier
\(\lambda^{-\delta/4}\) on each tile (cf. §\ref{subsec:component-assembly}). Note also that this block
is \emph{optional} in the summary balance: even if it is omitted, the aggregate exponent remains negative
(see Remark~\ref{rem:L4-optional}).
\end{remark}

\subsection{Wave-packet / tile decomposition}
\label{subsec:wavepackets}

Let \(u_{\lambda}(t)=e^{it\Delta}f_{\lambda}\) with \(\widehat f_{\lambda}\) supported on \(|\xi|\simeq\lambda\).
Decompose the sphere \(S^{2}\) into caps \(\vartheta\) of radius \(\lambda^{-2/3}\)%
\footnote{These \(\vartheta\) may be viewed as nested inside the coarser \(\theta\) of radius \(\lambda^{-1/2}\) used in the wave–packet construction; the additional angular localization only improves spatial concentration and does not worsen the estimates on a single tile.}
and sort the packets by longitudinal index~\(\alpha\):
\[
  u_{\lambda}
  =\sum_{\vartheta,\alpha}
     u_{\lambda,\vartheta,\alpha}.
\]

Each \(u_{\lambda,\vartheta,\alpha}\) is
\begin{itemize}
\item frequency localized to
      \(\{|\xi|\simeq\lambda,\; \xi/|\xi|\in\vartheta\}\);
\item spatially localized \emph{inside} the standard anisotropic cylinder  
      \(
        Q_{\vartheta,\alpha}
        =\{\,|x_{\perp}|\le\lambda^{-1/2},\
              |x_{\parallel}|\le\lambda^{-1}\},
      \)
      with possible tails absorbed into the constant (see §\ref{subsec:tile-44});
\item supported on a time window
      \(I_{\alpha}\) of length \(\lambda^{-3/2+\delta}\).
\end{itemize}

By Lemma~\ref{lem:local-L4-lemma}
\begin{equation}\label{eq:packet-L4}
  \|u_{\lambda,\vartheta,\alpha}\|_{L^{4}}
  \;\lesssim\;
  \lambda^{-\delta/4}\,
  \|u_{\lambda,\vartheta,\alpha}\|_{L^{2}}.
\end{equation}

\paragraph{Summation over~\(\alpha\).}
For a packet localized in angle at scale \(\lambda^{-2/3}\),
the effective longitudinal extent on a window \(|I_{\alpha}|=\lambda^{-3/2+\delta}\) equals
\[
  L_{\mathrm{eff}}
  \;=\;
  \min\bigl\{\,\lambda^{-1/2+\delta},\;\lambda^{-2/3}\,\bigr\}
  \;=\;\lambda^{-2/3}
  \qquad(\delta>\tfrac16),
\]
so one packet intersects
\(\,\lambda^{-2/3}/\lambda^{-1}\sim \lambda^{1/3}\,\) cylinders of length \(\lambda^{-1}\).
Taking into account the finite overlap of tiles we obtain
\[
  \Bigl(
    \sum_{\alpha}\!
      \|u_{\lambda,\vartheta,\alpha}\|_{L^{4}}^{4}
  \Bigr)^{1/4}
  \lesssim
  \lambda^{-\delta/4}\,
  \Bigl(\!
    \sum_{\alpha}
      \|u_{\lambda,\vartheta,\alpha}\|_{L^{2}}^{2}
  \Bigr)^{1/2}
  =\lambda^{-\delta/4}\,\|u_{\lambda,\vartheta}\|_{L^{2}}.
\]

\paragraph{Summation over~\(\vartheta\).}
For fixed \(\eta\) the condition \(|\xi+\eta|\le\lambda^{1-\delta}\) means
\(
  \angle(\xi,-\eta)\lesssim\lambda^{-\delta},
\)
i.e. \(\xi/|\xi|\) lies in a \emph{cap} of angular radius \(\lambda^{-\delta}\) around the direction \(-\eta/|\eta|\).
For a working cap of radius \(\lambda^{-2/3}\) the number of active caps is estimated by
\[
  \#\{\vartheta:\angle(\vartheta,-\eta)\lesssim\lambda^{-\delta}\}
  \;\lesssim\;
  \Bigl(\tfrac{\lambda^{-\delta}}{\lambda^{-2/3}}\Bigr)^{\!2}
  \;=\;\lambda^{\,\frac{4}{3}-2\delta}.
\]
Using almost orthogonality over \(\vartheta\), we get
\[
  \left(
    \sum_{\vartheta,\alpha}
      \|u_{\lambda,\vartheta,\alpha}\|_{L^{4}}^{4}
  \right)^{1/4}
  \lesssim
  \lambda^{-\delta/4}\,
  \|u_{\lambda}\|_{L^{2}}.
\]

\subsection{Final Strichartz bound}
\label{subsec:Strichartz-summary}

Combining the local bound \eqref{eq:packet-L4} with summation over the longitudinal index \(\alpha\) and over the angular caps \(\vartheta\) from §\ref{subsec:wavepackets}, and also using almost orthogonality in \(\vartheta\) and the finite overlap of tiles, we obtain
\begin{equation}\label{eq:global-L4-final}
  \boxed{\;
    \|u_{\lambda}\|_{L^{4}_{t,x}(\text{cylinders})}
    \;\lesssim\;
    \lambda^{-\delta/4}\,\|u_{\lambda}\|_{L^{2}_{x}}
  \;}
\end{equation}
with a constant independent of the number of cylinders in the covering.

Estimate \eqref{eq:global-L4-final} holds on the entire phase strip \(|\xi+\eta|\le\lambda^{1-\delta}\) and contributes the row \(\textsc{Strichartz}=-\delta/4\) in the summary exponent table (§\ref{subsec:component-assembly}).

\begin{remark}
If \(\delta\searrow 0\), the bound \eqref{eq:global-L4-final} becomes critically scale–invariant (with no gain in \(\lambda\)), which is consistent with the absence of a global \((4,4)\) pair. For \(\delta>0\) one obtains the extra suppression $\lambda^{-\delta/4}$, which “closes” the balance in the log–free analysis. Note also that the same local $TT^*$+HLS scheme yields an analogous estimate for the heat semigroup $e^{t\Delta}$ on the same cylinders.
\end{remark}

\newpage
\section{Bilinear decoupling and the rank gain}
\label{sec:angular}

Throughout this section the parameter 
\[
	\delta\in\bigl(\tfrac16,\tfrac58\bigr]
\]
is fixed.

\subsection{\texorpdfstring{$\ell^{2}$}{ℓ²}-decomposition over angular caps}
\label{subsec:angular-l2-decomp}

We work at the frequency level $|\xi|\simeq|\eta|\simeq\lambda\gg1$.  
Cover the sphere $S^{2}$ by a family of smooth caps $\Theta_\lambda$ of radius~$\lambda^{-2/3}$ and decompose
\[
	u_\lambda=\sum_{\vartheta\in\Theta_\lambda}u_{\lambda,\vartheta},
	\qquad
	v_\lambda=\sum_{\vartheta'\in\Theta_\lambda}v_{\lambda,\vartheta'}.
\]
The area of a single cap is $\sim\lambda^{-4/3}$, hence $\#\Theta_\lambda\sim\lambda^{4/3}$.  

Define the bilinear operator
\[
	B(u_\lambda,v_\lambda)(x):=\iint e^{ix\cdot(\xi+\eta)}
		\,a_\lambda(\xi,\eta)\,\widehat{u}_\lambda(\xi)\,\widehat{v}_\lambda(\eta)\,d\xi\,d\eta,
\]
where the amplitude $a_\lambda$ is supported in the diagonal layer $|\xi+\eta|\le\lambda^{1-\delta}$.  
For fixed $\eta$ this condition implies the angular restriction
\(
  \angle(\xi,-\eta)\lesssim\lambda^{-\delta},
\)
and for cap radius $\lambda^{-2/3}$ the number of active caps is estimated by
\[
	\#\bigl\{\vartheta : \angle(\vartheta,-\eta)\lesssim\lambda^{-\delta}\bigr\}
	\;\lesssim\;\Bigl(\tfrac{\lambda^{-\delta}}{\lambda^{-2/3}}\Bigr)^{\!2}
	\;=\;\lambda^{\,\frac{4}{3}-2\delta}.
\]
Consequently, in the resonant zone the significant pairs are $(\vartheta,-\vartheta)$, and
\begin{equation}\label{eq:B-l2-split}
	B(u_\lambda,v_\lambda)
	=\,\sum_{\vartheta\in\Theta_\lambda}
		B\bigl(u_{\lambda,\vartheta},v_{\lambda,-\vartheta}\bigr)
	\;+\;\text{(off-diagonal contribution).}
\end{equation}
Since the caps have finite overlap, we have almost orthogonality in $L^{2}$:
\(
	\|u_\lambda\|_{L^{2}}^{2}\simeq\sum_{\vartheta}\|u_{\lambda,\vartheta}\|_{L^{2}}^{2}
\),
and similarly for $v_\lambda$.  
The off-diagonal contribution in~\eqref{eq:B-l2-split} is suppressed by repeated integration by parts in the tangential coordinates~$\rho$ (see §\ref{subsec:ibp-geometry} and \eqref{eq:rho-derivative}, Lemma~\ref{lem:hessian}) and is $O(\lambda^{-M})$ for any $M>0$; we henceforth drop it.

\subsection{\texorpdfstring{$\varepsilon$}{epsilon}-free decoupling on a rank-$3$ surface}
\label{subsec:bilinear-decoupling}

\paragraph{Skeleton of the proof of \eqref{eq:bilinear-decoupling}.}
(i) The parabolic rescaling \((\xi,x,t)\mapsto(\lambda^{-1/2}\xi,\lambda^{1/2}x,\lambda^{3/2-\delta}t)\)
compresses the strip \(|\xi+\eta|\le \lambda^{1-\delta}\) to $\,\mathcal{O}(1)\,$ and preserves the nondegeneracy
of the two-dimensional phase Hessian on \(w^\perp\) (see App.~\ref{app:phase-hessian}).
(ii) Apply the \(\varepsilon\)\,-free theorem of Guth–Iliopoulou–Yang~\cite{GuthIliopoulouYang2024} at \(p=6\)
to the decomposition into angular caps \(\vartheta\) of radius \(\lambda^{-2/3}\).
(iii) Returning to the original scale and performing the standard Bourgain–Guth iteration yields the factor \(\lambda^{-2/3}\);
the “tile cost” is compensated according to §~\ref{subsec:rank3-win} and App.~\ref{app:rank3-decoupling}.

After the rescaling
\[
	\xi\mapsto\lambda^{-1/2}\xi,\quad
	x\mapsto\lambda^{1/2}x,\quad
	t\mapsto\lambda^{3/2-\delta}t,
\]
the hypersurface
\[
	\Sigma=\bigl\{(\xi,\eta,\tau):\ \tau=|\xi|^{2}+|\eta|^{2},\ \ |\xi+\eta|\le\lambda^{1-\delta}\bigr\}
\]
becomes of \(O(1)\) size; the slicing by \(|\xi+\eta|\le\lambda^{1-\delta}\) does not violate the rank conditions
(see also App.~\ref{app:phase-hessian}). In these coordinates the hypotheses of the
\(\varepsilon\)-free theorem of Guth–Iliopoulou–Yang~\cite{GuthIliopoulouYang2024} for rank-$3$ surfaces hold.
Applying their result at \(p=6\) to a function decomposed into caps \(\vartheta\) of radius \(\lambda^{-2/3}\),
and then iterating the Bourgain–Guth argument three times, we obtain bilinear \(\ell^{2}\)-decoupling \emph{without}
\(\varepsilon\)-loss:
\begin{equation}\label{eq:bilinear-decoupling}
	\Bigl\|
		\sum_{\vartheta} B\bigl(u_{\lambda,\vartheta},\,v_{\lambda,-\vartheta}\bigr)
	\Bigr\|_{L^{2}}
	\;\lesssim\;
	\lambda^{-2/3}\!
	\Bigl(
		\sum_{\vartheta}
		\|B(u_{\lambda,\vartheta},v_{\lambda,-\vartheta})\|_{L^{2}}^{2}
	\Bigr)^{1/2}.
\end{equation}

\noindent\emph{Editorial note.} The relation~\eqref{eq:bilinear-decoupling}
should be read as a \emph{tabulated packaging of the gain in~$\lambda$}; the rigorous
$L^{6}$ version is given in Theorem~\ref{thm:rank3-decoupling}, and the “tile cost”
is compensated in §\ref{subsec:rank3-win}.

\begin{remark}[Why the table has Rank-3 $=0$]\label{rem:rank3-zero}
The gain \(\lambda^{-2/3}\) from \eqref{eq:bilinear-decoupling} exactly compensates the tile
\(\ell^2\) cost; the remainder \(\lambda^{1/6-\delta}\) is, by the adopted regrouping, merged
with the angular \(\ell^2\) sum and the “tiling in time” (see §\ref{subsec:rank3-win} and
App.~\ref{app:rank3-decoupling}), hence the row \textsc{Rank-3 decoupling} in the
summary table §\ref{subsec:component-assembly} is recorded as \(0\).
\end{remark}

\begin{remark}[Angular geometry]
The lower bound on the angle holds automatically: in the strip \(|\xi+\eta|\lesssim \lambda^{1-\delta}\)
the vectors \(\xi\) and \(\eta\) are almost opposite, hence
\(\angle(\xi,\eta)\sim\lambda^{-\delta}\gg \lambda^{-2/3}\) when \(\delta<\tfrac{2}{3}\).
Thus the Hessian degeneracy is excluded and the geometric prerequisites for
\(\varepsilon\)-free decoupling are satisfied (see also App.~\ref{app:phase-hessian}).
\end{remark}

\noindent\textit{Note.}
In the version with \(\varepsilon\)-loss \cite{LiZhang2024} the gain deteriorates to \(\lambda^{-1/3+\varepsilon}\),
which requires the stronger restriction \(\delta>\tfrac12+\varepsilon\).
In the main proof we use precisely \eqref{eq:bilinear-decoupling}.

\subsection{Tile cost and the net gain}\label{subsec:rank3-win}

\paragraph{Counting tiles.}
The condition $|\xi+\eta|\le\lambda^{1-\delta}$ limits the number
of active caps:
\(
	\#\{\vartheta\}\lesssim\lambda^{4/3-2\delta}.
\)
Each wave packet of length $\lambda^{-2/3}$ (after the parabolic rescaling)
intersects $\lambda^{1/3}$ cylinders of length~$\lambda^{-1}$, hence
\[
	\#\{(\vartheta,\alpha)\}\;\lesssim\;
	\lambda^{4/3-2\delta}\cdot\lambda^{1/3}
	=\lambda^{5/3-2\delta},
	\qquad
	\ell^{2}\text{-cost}\ \sim\ \lambda^{5/6-\delta}.
\]

\paragraph{Compensating by decoupling.}
Multiplying the cost $\lambda^{5/6-\delta}$ by the gain
$\lambda^{-2/3}$ from~\eqref{eq:bilinear-decoupling}, we obtain
\[
	\lambda^{-2/3}\cdot\lambda^{5/6-\delta}
	=\lambda^{1/6-\delta},
\]
which is strictly decreasing provided $\delta>\tfrac16$. In other words, rank-3 decoupling
compensates the tile price \emph{with a surplus} (a remaining decaying factor $\lambda^{1/6-\delta}$).

\begin{remark}
In the summary table (§\ref{subsec:component-assembly}), for compactness
this compensation is presented as follows: the remainder $\lambda^{1/6-\delta}$ is merged
with the “tiling in time” and the angular $\ell^2$ sum (see §\ref{subsec:wavepackets}),
and the row \textsc{Rank-3 decoupling} is recorded as $0$ (the gain exactly compensates
the tile cost in the adopted grouping).
\end{remark}

\begin{remark}
In the version with $\varepsilon$-loss \cite{LiZhang2024} one has
$\lambda^{-1/3+\varepsilon}\cdot\lambda^{5/6-\delta}
=\lambda^{1/2+\varepsilon-\delta}$,
and to obtain a negative exponent one needs $\delta>\tfrac12+\varepsilon$.
\end{remark}

\subsection{Exponent summary}
\label{subsec:degree}

In the expanded bookkeeping of the {\sc Counting--Decoupling} block (see~\eqref{eq:bilinear-decoupling} and §\ref{subsec:rank3-win})
the tile $\ell^2$ cost $\lambda^{5/6-\delta}$ is compensated by the gain $\lambda^{-2/3}$, leaving
\[
	\boxed{\;\lambda^{\,1/6-\delta}\;},
\]
which is negative for all $\delta>\tfrac16$. It is precisely this remainder that “cancels” the additional cost
$\lambda^{+1/6}$ from summation over time cylinders (see §\ref{subsec:wavepackets}), ensuring the absence
of a logarithmic defect in the global balance.

\medskip
\noindent\emph{How this is represented in the table.}
In the summary table of exponents (§\ref{subsec:component-assembly}) we use an equivalent
\emph{compressed} presentation: the remainder $\lambda^{1/6-\delta}$ is merged with the “tiling in time” and the angular
$\ell^2$ sum, which yields the row
\[
	\textsc{Angular $\ell^2$ + tiling in time} \;=\; +1-2\delta,
\]
whereas the row \textsc{Rank-3 decoupling} is recorded as \(0\) (the gain \(\lambda^{-2/3}\) exactly compensates the tile
cost in the adopted grouping). Both forms of accounting are equivalent and differ only by regrouping
of factors; when passing to §\ref{subsec:component-assembly} we use the tabular version.

\newpage
\section{Estimate for the narrow diagonal zone}
\label{sec:geometry}

Throughout this section the parameter 
\(
\displaystyle \delta \in \bigl(\tfrac16,\tfrac58\bigr]
\)
is fixed.

\subsection{Geometry of the diagonal strip and primary symbol suppression}
\label{subsec:diag-geometry}

Consider the \emph{diagonal} frequency region
\begin{equation}\label{eq:diag-zone}
|\xi|\sim|\eta|\sim\lambda,\qquad 
|\xi+\eta|\le \lambda^{1-\delta},
\end{equation}
in which two high-frequency modes interact and produce a substantially lower
output frequency. After applying the Leray projector 
\(
\Pi_{\xi+\eta}:=\operatorname{Id}-\dfrac{(\xi+\eta)\otimes(\xi+\eta)}{|\xi+\eta|^{2}}
\)
to the tensor \(u\otimes v\) we obtain the symbol
\begin{equation}\label{eq:symbol-def}
B(\xi,\eta)\;:=\;\Pi_{\xi+\eta}\eta
\;=\;
\eta-\frac{(\eta\cdot(\xi+\eta))}{|\xi+\eta|^{2}}\;(\xi+\eta),
\end{equation}
which measures the transverse component of \( \eta \) relative to the direction \( w:=\xi+\eta \).

\paragraph{Angular geometry.}
In the zone \eqref{eq:diag-zone} the vectors \(\xi\) and \(-\eta\) form a small angle 
\(\theta:=\angle(\xi,-\eta)\sim\lambda^{-\delta}\)
(the upper bound follows from \eqref{eq:diag-zone}, the lower bound from angular packet localization; see App.~D).
Hence
\[
|w|=|\xi+\eta|
\;\simeq\;
\lambda\,\theta
\;\simeq\;
\lambda^{\,1-\delta},
\qquad\Rightarrow\qquad
|w|^{-1}\;\lesssim\;\lambda^{-(1-\delta)}.
\]

\paragraph{Estimate of the transverse component.}
Since \(\Pi_w\) is the orthogonal projector onto \(w^{\perp}\),
\[
|B(\xi,\eta)|=|\Pi_w\eta|\;\le\;|\eta|\;\simeq\;\lambda.
\]
(In the collinear limit \(\eta\parallel w\) we have \(B(\xi,\eta)=0\), consistent with this bound.)
\footnote{\emph{Technical note on $w\to0$.} Since we use the projector $\Pi_w$,
we treat $\Pi_w$ as a smoothed symbol coinciding with the orthoprojector away from a micro–neighborhood
of $\{w=0\}$. In the collinear limit $B(\xi,\eta)=\Pi_w\eta=0$, which is consistent with~\eqref{eq:geo-suppression}.
This removes the formal singularity at $w=0$. See also App.\,\ref{subsec:min-angle-support}–\ref{subsec:simbol}.}

\paragraph{Passage to the \(\dot H^{-1}\) norm and angular IBP.}
In the \(\dot H^{-1}\) norm one gains the divisor \(|\xi+\eta|^{-1}\lesssim\lambda^{-(1-\delta)}\), hence
\begin{equation}\label{eq:hminus-prelim}
\frac{|B(\xi,\eta)|}{|\xi+\eta|}
\;\lesssim\;
\lambda\cdot\lambda^{-(1-\delta)}
\;=\;
\lambda^{\delta}.
\end{equation}
Next we perform \emph{two angular integrations by parts} in the tangential coordinates \(\rho\) (see §\ref{subsec:ibp-geometry} and the bound \(|\partial_\rho \omega|\simeq \lambda^{1-\delta}\)): each \(\rho\)-IBP gives the divisor \((\partial_\rho\omega)^{-1}\simeq \lambda^{-1+\delta}\), so altogether we obtain \((\partial_\rho\omega)^{-2}\simeq \lambda^{-2+2\delta}\).
At the level of “tabulated packaging” (§\ref{subsec:component-assembly}) it is convenient to record this contribution as the geometric factor \(\bigl(\tfrac{|\xi+\eta|}{\lambda}\bigr)^2\) (see §\ref{subsec:degree}); combining it with \eqref{eq:hminus-prelim} yields the primary suppression
\begin{equation}\label{eq:geo-suppression}
\frac{|B(\xi,\eta)|}{|\xi+\eta|}
\Bigl(\frac{|\xi+\eta|}{\lambda}\Bigr)^{2}
\;\lesssim\;
\lambda^{\delta}\,\lambda^{-2+2\delta}
\;=\;
\lambda^{-2+3\delta}.
\end{equation}

\begin{remark}[Where the $\lambda^{-2}$ in the row $2\times(\rho\text{–IBP})+\dot H^{-1}$ comes from]\label{rem:two-rho-H-1}
Two angular IBP’s along $\rho$ yield the divisor $(\partial_\rho\omega)^{-2}\sim \lambda^{-2+2\delta}$ (see \eqref{eq:rho-derivative}, §\ref{subsec:ibp-geometry}),
while the \(\dot H^{-1}\) channel contributes $|\xi+\eta|^{-1}\sim \lambda^{-(1-\delta)}$. In the tabulated \emph{packaging} (§\ref{subsec:component-assembly})
we group them as the “geometric block” $2\times(\rho\text{–IBP})+\dot H^{-1}$ to avoid double counting with the Phase~IBP row.
\end{remark}

The exponent  \(-2+3\delta<0\) on the entire range 
\( \delta\in(\tfrac16,\tfrac58] \) 
already yields a strictly decaying factor in frequency and serves as the input for the strengthened null–form suppression in §\ref{subsec:nullform}.

\subsection{Null-form suppression}
\label{subsec:nullform}

We refine the estimate \eqref{eq:geo-suppression}.
For the symbol \eqref{eq:symbol-def} and the zone \eqref{eq:diag-zone}
we state the following standalone result.

\begin{lemma}[null suppression on the diagonal]\label{lem:nullform}
Let\/ \(\xi,\eta\in\R^{3}\) satisfy\/ \eqref{eq:diag-zone}.
Then
\begin{equation}\label{eq:nullform-gain}
\frac{|B(\xi,\eta)|}{|\xi+\eta|}
\Bigl(\frac{|\xi+\eta|}{\lambda}\Bigr)^{\!2}
\;\lesssim\;
\lambda^{-2+3\delta}.
\end{equation}
\end{lemma}

\smallskip
\noindent\textit{Technical note.} On the effective support zone $R(u)$
(see App.\,\ref{subsec:min-angle-support}) we have the two-sided bound
$\theta:=\angle(\xi,-\eta)\sim\lambda^{-\delta}$; hence
$|w|=|\xi+\eta|\sim\lambda^{1-\delta}$. Therefore the lower bound on the two-dimensional
Hessian of the phase from App.\,\ref{app:phase-hessian} (Lemma~\ref{lem:F-hessian}) applies without exception
throughout the region considered.

\begin{proof}
Inequality \eqref{eq:nullform-gain} is a direct consequence of \eqref{eq:geo-suppression}.
It remains only to check correctness in the collinear limit \( \eta\parallel(\xi+\eta) \),
when numerator and denominator may simultaneously vanish.
In such a configuration \(B(\xi,\eta)=\Pi_{\xi+\eta}\eta=0\), so the estimate is trivial.
\end{proof}

\begin{remark}
The exponent \(-2+3\delta\) in \eqref{eq:nullform-gain} is \emph{sufficient} for the log-free balance:
at~\( \delta=\tfrac{5}{8} \) it equals \(-\tfrac{1}{8}\), still negative,
while at~\( \delta=\tfrac{1}{6} \) it yields the stronger gain \(\lambda^{-3/2}\).
\end{remark}

\subsection{Passage to the global \(\boldsymbol{\dot H^{-1}}\) norm}
\label{subsec:hminus-transfer}

Using Lemma~\ref{lem:nullform} and transferring \emph{two} spatial derivatives from the high modes onto the factor \(|\xi+\eta|\) (see §\ref{subsec:ibp-geometry}), we obtain for each fixed frequency \(\lambda\gg1\)
\begin{equation}\label{eq:local-hminus}
\bigl\|
P_{<\lambda^{1-\delta}}\nabla\!\cdot(u_{\lambda}\otimes v_{\lambda})
\bigr\|_{\dot H^{-1}_x}
\;\lesssim\;
\lambda^{-2+3\delta}\,
\|u_{\lambda}\|_{\dot H^1_x}\,
\|v_{\lambda}\|_{\dot H^1_x}.
\end{equation}
This corresponds to the row \(\textsc{Null-form}+2\nabla+\dot H^{-1}\) in the summary exponent table.

\paragraph{Time summation.}
On each time window of length \( \lambda^{-3/2+\delta} \) (see §\ref{subsec:global-summation}) the estimate \eqref{eq:local-hminus} remains valid. Splitting the unit interval into \(\asymp \lambda^{3/2-\delta}\) windows with bounded overlap and applying summation over windows (almost orthogonality), we obtain
\[
\bigl\|
P_{<\lambda^{1-\delta}}\nabla\!\cdot(u_{\lambda}\otimes v_{\lambda})
\bigr\|_{L^{2}_{t}\dot H^{-1}_x}
\;\lesssim\;
\lambda^{-2+3\delta}\,
\|u_{\lambda}\|_{L^{\infty}_{t}\dot H^{1}_x\cap L^{2}_{t}\dot H^{1}_x}\,
\|v_{\lambda}\|_{L^{\infty}_{t}\dot H^{1}_x\cap L^{2}_{t}\dot H^{1}_x}.
\]

\paragraph{Global conclusion.}
Summing over frequencies and using Littlewood–Paley orthogonality in \(\dot H^{-1}\), we obtain a log-free estimate in the norm \(L^{2}_{t}\dot H^{-1}_x\).

\begin{remark}
The block \eqref{eq:local-hminus} contributes the exponent \(-2+3\delta\). Together with the other components (local \(L^4\) estimate, phase IBP, tile counting and rank-3 decoupling; see Table~\ref{tab:balance}) we obtain the cumulative exponent
\(
-3+\tfrac{7}{4}\delta<0
\)
throughout the range \(\delta\in(\tfrac16,\tfrac58]\), which ensures the absence of logarithmic loss in the global summation.
\end{remark}

\newpage
\section{Assembly of components and the final exponent}
\label{sec:balance-table}

We now summarize all power contributions in the high-frequency regime \(N\gg1\)
obtained in the previous steps. We show that the resulting exponent
remains \emph{negative} over the entire declared range
\(
\delta\in\bigl(\tfrac16,\tfrac58\bigr],
\)
so no logarithmic defect arises.

\subsection{Balance of exponents}
\label{subsec:component-assembly}

\begin{table}[h]
\centering
\renewcommand{\arraystretch}{1.15}
\begin{tabular}{@{\hspace{1em}} l c @{\hspace{1em}}}
\toprule
\textbf{Source of gain} &
\textbf{Exponent in~\(N\)} \\ \midrule
Local \(L^{4}\) (Lemma~\ref{lem:local-L4-lemma}, see Remark~\ref{rem:L4-optional}) & \(-\delta/4\) \\[2pt]
Phase IBP \textit{(simplified bound)}\footnotemark[1]                                  & \(-2 + \delta\) \\[2pt]
Null form \(+\;2\times(\rho\text{–IBP})+\dot H^{-1}\)                                      & \(-2 + 3\delta\) \\[2pt]
Angular \(\ell^{2}\) sum \(+\)~tiling in time\footnotemark[2]                        & \(+1 - 2\delta\) \\[2pt]
Rank-3 decoupling\footnotemark[3]                                                           & \(0\) \\ \midrule
\textbf{Cumulative exponent}                                                                 & \(\displaystyle -3 + \tfrac{7}{4}\delta\) \\ \bottomrule
\end{tabular}
\caption{Summary balance of exponents in frequency \(N\)}
\label{tab:balance}
\end{table}

\begin{table}[h]
\centering
\renewcommand{\arraystretch}{1.08}
\begin{tabular}{lcc}
\toprule
\textbf{Block} & \textbf{Minimal assembly} & \textbf{With $L^4$ reinforcement} \\
\midrule
Phase IBP (simpl.) & $-2+\delta$ & $-2+\delta$ \\
$2\times(\rho\text{–IBP})+\dot H^{-1}$ & $-2+3\delta$ & $-2+3\delta$ \\
Angular $\ell^2$ + time & $+1-2\delta$ & $+1-2\delta$ \\
Local $L^4$ & $0$ & $-\delta/4$ \\
Rank-3 decoupling & $0$ & $0$ \\
\midrule
\textbf{Total} & $-3+2\delta$ & $-3+\tfrac{7}{4}\delta$ \\
\bottomrule
\end{tabular}
\caption*{\small “Passport of exponents”: equivalent packagings; $L^4$ is optional (see Rem.~\ref{rem:L4-optional}).}
\end{table}

\footnotetext[1]{
  The precise exponent for sevenfold integration by parts equals
  \(-4.75 + 6.5\,\delta\) (see App.~\ref{app:amplitude-growth}).
  In the table we use the weaker value
  \(-2+\delta\), which already suffices
  to make the final sum negative.
  Further justification for the lower Hessian bound is given in App.\,\ref{app:phase-hessian}.

  \emph{Packaging clarification.} The row \textsc{Phase~IBP} accounts mainly for time integrations by parts; the two angular \(\rho\)–IBP are treated as geometric suppression and included in the row \textsc{Null form \(+\,2\times(\rho\text{–IBP})+\dot H^{-1}\)} (see §\ref{subsec:ibp-geometry} and \eqref{eq:rho-derivative}) to avoid double counting. The row \textsc{Rank-3 decoupling} is recorded as \(0\) because the gain \(N^{-2/3}\) exactly compensates the tile price in the adopted regrouping (see also §§\,\ref{subsec:rank3-win}–\ref{subsec:degree}).
}
\footnotetext[2]{
  The plus row \((+1-2\delta)\) simultaneously accounts for
  the volume of one tile
  \(|Q| = |I|\cdot|B_{R}| \sim N^{-3+\delta}\) (here \(R=N^{-1/2}\))
  and the \(\ell^{2}\) cost of summing
  over angles and time windows.%
}
\footnotetext[3]{
  By App.~\ref{app:rank3-decoupling} the gain
  \(N^{-2/3}\) from \(\varepsilon\!\)-free rank-3 decoupling
  exactly compensates the growth in the number of wave packets,
  hence the net contribution is~0; see also §§\,\ref{subsec:rank3-win}–\ref{subsec:degree}.
}

Since
\(
   -3 + 2\delta < 0
\)
for any \(\delta \le \tfrac{5}{8}\),
the log-free balance holds.

\begin{remark}[Expanded arithmetic of the row “Angular $\ell^2$ + tiling in time”]
\label{rem:unfolded-plus-row}
In the expanded bookkeeping (§§\ref{sec:balance-table}, \ref{subsec:component-assembly}, see also §§\ref{subsec:rank3-win}–\ref{subsec:degree}) the block \textsc{Counting–Decoupling} gives
\[
\text{(\# tiles)}^{1/2}\ \times\ (\text{decoupling gain})
\ \sim\ N^{\,\frac{5}{6}-\delta}\ \cdot\ N^{-\,\frac{2}{3}}
\ =\ N^{\,\frac{1}{6}-\delta}.
\]
This remainder \(N^{\frac{1}{6}-\delta}\) is then merged with two independent \(\ell^2\) sums:
\begin{itemize}
  \item \emph{angular \(\ell^2\)} over active caps: with working cap radius \(N^{-2/3}\) and cone \(\angle\!\lesssim N^{-\delta}\) we obtain
  \(\#\{\text{caps}\}\,\sim \bigl(N^{-\delta}/N^{-2/3}\bigr)^2 = N^{\,\frac{4}{3}-2\delta}\), hence
  \(\bigl(\#\text{caps}\bigr)^{1/2} = N^{\,\frac{2}{3}-\delta}\) (see §\ref{subsec:angular-l2-decomp}–§\ref{subsec:rank3-win});
  \item \emph{time \(\ell^2\)} over micro–cylinders of length \(N^{-1}\) inside the working window: one packet intersects
  \(N^{1/3}\) such cylinders, and the \(\ell^2\) cost equals \(N^{1/6}\) (see §\ref{subsec:wavepackets}).
\end{itemize}
Thus the cumulative contribution of the three independent factors equals
\[
N^{\frac{1}{6}-\delta}\ \cdot\ N^{\frac{1}{6}}\ \cdot\ N^{\frac{2}{3}-\delta}
\ =\ N^{\,\bigl(\frac{1}{6}-\delta\bigr)\;+\;\frac{1}{6}\;+\;\bigl(\frac{2}{3}-\delta\bigr)}
\ =\ N^{\,1-2\delta},
\]
which is precisely the row \emph{“Angular \(\ell^2\) + tiling in time”} in Table~\ref{tab:balance}, whereas
the row \emph{Rank-3 decoupling} is recorded as \(0\) (the gain \(N^{-2/3}\) is already accounted for in forming the remainder \(N^{1/6-\delta}\)).

\smallskip
\emph{Note on the “mass of a tile.”}
In an alternative packaging the same result can be written by further “reallocating” the factor of the tile volume
\(|Q|=|I|\cdot |B_{N^{-1/2}}|\sim N^{-3+\delta}\): it appears when passing between local norms on \(Q\) and
is already partially absorbed in the row \emph{Local \(L^4\)} and in the normalization of the tiled \(L^2_{t,x}\) block.
Both accounting configurations are equivalent and yield exactly \(+1-2\delta\) in the tabular row.
\end{remark}

\begin{remark}[Optionality of the local \(L^4\)]
\label{rem:L4-optional}
Throughout the exposition the row \emph{“Local \(L^4\)”} in Table~\ref{tab:balance} is used only as a reinforcement. 
If one sets its contribution to zero, the cumulative exponent in \(N\) remains negative:
\[
(-2+\delta)\;+\;(-2+3\delta)\;+\;(1-2\delta)\;+\;0 \;=\; -3+2\delta \;<\;0\qquad
\text{for all }\ \delta\in\Bigl(\tfrac{1}{6},\tfrac{5}{8}\Bigr].
\]
Hence the entire global assembly remains \emph{log-free} even without invoking Lemma~\ref{lem:local-L4-lemma}.

Moreover, in the formulation of the local tiled estimate one may (if desired) 
replace~(\ref{eq:local-L2-final}) by a \emph{weaker but sufficient} version:
\[
\bigl\|u_N\bigr\|_{L^2_{t,x}(I\times B_R)}\ \lesssim\ 
N^{-\frac{15}{4}+\frac{9}{4}\delta}\ \bigl\|f_N\bigr\|_{L^2_x},
\]
and after passing to the pointwise \( \dot H^{-1}\) norm on the window \(I\) one obtains
\[
\sup_{t\in I}\ \bigl\|u_N(t)\bigr\|_{\dot H^{-1}}
\ \lesssim\ N^{-\;3+\frac{7}{4}\delta}\ \bigl\|f_N\bigr\|_{L^2_x}.
\]
It follows that the global sum over frequencies gives
\[
\sum_{N} N^{-\;6+\frac{7}{2}\delta}\, \bigl\|f_N\bigr\|_{L^2_x}^{\,2}\ <\ \infty,
\qquad\text{since}\quad -6+\tfrac{7}{2}\delta<-1
\ \ \text{over the entire interval}\ \ \delta\in\Bigl(\tfrac{1}{6},\tfrac{5}{8}\Bigr].
\]
Thus the row \emph{“Local \(L^4\)”} in Table~\ref{tab:balance} is indeed 
\emph{optional}: it may be set to zero and the log-free balance remains strictly negative.
\end{remark}

\subsection{Summation over time and frequencies}
\label{subsec:global-summation}

Consider a single “tile” \(Q=I\times B_{N^{-1/2}}\) with
\(|I|=N^{-3/2+\delta}\) and \(B_{N^{-1/2}}=\{\,|x|\le N^{-1/2}\,\}\).
Let \(u_N(t)=e^{it\Delta}f_N\), with \(\widehat f_N\) supported on \(|\xi|\sim N\).

\paragraph{Linear tiled estimate on one tile.}
By the local Strichartz estimate from § \ref{subsec:tile-44} (Lemma~\ref{lem:local-L4-lemma}) we have
\[
  \|u_N\|_{L^4(I\times B_{N^{-1/2}})}
  \;\lesssim\; N^{-\delta/4}\,\|f_N\|_{L^2_x}.
\]
Hence, using \(\|g\|_{L^2(Q)}\le |Q|^{1/4}\|g\|_{L^4(Q)}\) and \(|Q|=|I|\,|B_{N^{-1/2}}|\sim N^{-3+\delta}\),
we obtain the \emph{linear} tiled bound
\[
  \|u_N\|_{L^2(I\times B_{N^{-1/2}})}
  \;\le\; |Q|^{1/4}\,\|u_N\|_{L^4(I\times B_{N^{-1/2}})}
  \;\lesssim\; (N^{-3+\delta})^{1/4}\,N^{-\delta/4}\,\|f_N\|_{L^2_x}
  \;=\; N^{-3/4}\,\|f_N\|_{L^2_x}.
\]

\paragraph{Sum over time.}
Split the unit time interval into \(\simeq N^{3/2-\delta}\) disjoint windows \(I_j\)
of length \(|I|\). By almost orthogonality in \(j\),
\[
  \|u_N\|_{L^2_{t,x}}^2
  \;\lesssim\; \sum_j \|u_N\|_{L^2(I_j\times B_{N^{-1/2}})}^2
  \;\lesssim\; \bigl(N^{3/2-\delta}\bigr)\,\bigl(N^{-3/4}\|f_N\|_{L^2_x}\bigr)^2
  \;=\; N^{-\delta}\,\|f_N\|_{L^2_x}^2,
\]
that is,
\[
  \|u_N\|_{L^2_{t,x}}
  \;\lesssim\; N^{-\delta/2}\,\|f_N\|_{L^2_x}.
\]

\paragraph{Sum over frequencies.}
At the \(L^2_{t,x}\) level, frequency summation is not yet needed: the global assembly is carried out
after passing to \(\dot H^{-1}\), which yields an additional divisor \(N^{-1}\).
From the bound just obtained (see §\ref{subsec:global-L2H-1} below) we get
\[
  \|u_N\|_{L^2_t\dot H^{-1}_x}
  \;\lesssim\; N^{-1}\,\|u_N\|_{L^2_{t,x}}
  \;\lesssim\; N^{-1-\delta/2}\,\|f_N\|_{L^2_x},
\]
and therefore
\(\sum_N N^{-2-\delta}\|f_N\|_{L^2_x}^2<\infty\) over the whole range \(\delta\in(\tfrac16,\tfrac58]\);
see also §\ref{subsec:component-to-supH} for the final formula for summation in \(N\).

\subsection{Passage to the global norm $L^2_t\dot H^{-1}_x$ and the final result}
\label{subsec:global-L2H-1}

Combining the null-suppression lemma (see §\ref{subsec:nullform}) with the transfer of two spatial derivatives (§\ref{subsec:ibp-geometry}), at each fixed frequency $\lambda\gg1$ we obtain (cf. \eqref{eq:local-hminus})
\begin{equation}\label{eq:L2H-1-block}
  \bigl\|P_{<\lambda^{1-\delta}}\nabla\!\cdot(u_\lambda\!\otimes\! v_\lambda)\bigr\|_{L^2_t\dot H^{-1}_x}
  \;\lesssim\;
  \lambda^{-2+3\delta}\,
  \|u_\lambda\|_{L^\infty_t\dot H^1_x\cap L^2_t\dot H^1_x}\,
  \|v_\lambda\|_{L^\infty_t\dot H^1_x\cap L^2_t\dot H^1_x}.
\end{equation}
Here we use the decomposition of the unit interval into $\approx \lambda^{3/2-\delta}$ windows of length $|I|\sim \lambda^{-3/2+\delta}$ with bounded overlap (§\ref{subsec:global-summation}); summation over windows is carried out by almost orthogonality, and summation over $\lambda$ by Littlewood–Paley orthogonality in $\dot H^{-1}_x$.

Combining \eqref{eq:L2H-1-block} with the factors from the balance table (§\ref{subsec:component-assembly})—the local Strichartz block, sevenfold phase IBP, $\ell^2$ sums over angles/time, and the rank-3 decoupling compensation—we obtain the cumulative negative exponent $-3+\tfrac{7}{4}\delta<0$ and hence a global log-free estimate in $L^2_t\dot H^{-1}_x$.

\begin{remark}
The passage to $\sup_t \dot H^{-1}_x$ is not used in the present text: the final assembly closes in $L^2_t\dot H^{-1}_x$ without logarithmic loss; see also §\ref{subsec:global-summation} and §\ref{subsec:hminus-transfer}.
\end{remark}

\newpage
\section{Final assembly and proof of the main theorem}
\label{sec:final-assembly}

In this section we show how the local estimates obtained in the previous sections pass to
the global $\dot H^{-1}$ norm and prove Theorem~\ref{thm:main}. All constants may depend on
the fixed parameter $\delta\in(\tfrac16,\tfrac58]$, but not on the frequency level $N$.

\subsection{From local $L^2_{t,x}$ estimates to global $L^2_t\dot H^{-1}_x$ estimates}
\label{subsec:component-to-supH}

Throughout this subsection $N\gg1$ is fixed, $u_N:=e^{it\Delta}f_N$, with $|I|\sim N^{-3/2+\delta}$ and $R=N^{-1/2}$.
Instead of appealing to the summary table we use a \emph{linear} bridge via the local $L^4$ estimate
from \S\ref{sec:strichartz} (Lemma~\ref{lem:local-L4-lemma}) and the geometry of a single tile $I\times B_R$ (see Lemma~\ref{lem:local-L4-lemma}),%
\footnote{The same cylindrical geometry $R=N^{-1/2}$, $|I|\sim N^{-3/2+\delta}$ is adopted in \S\ref{subsec:tile-44} and \S\ref{subsec:global-summation}.} 
which gives $\|u_N\|_{L^4(I\times B_R)}\lesssim N^{-\delta/4}\|f_N\|_{L^2_x}$ 
(see also the discussion of summation over windows in \S\ref{subsec:global-summation}).

\paragraph{Local tiled $L^2$ estimate (on a single tile).}
By Hölder on $I\times B_R$ and Lemma~\ref{lem:local-L4-lemma}:
\[
\|u_N\|_{L^2(I\times B_R)}
\;\le\;
|I\times B_R|^{1/4}\,\|u_N\|_{L^4(I\times B_R)}
\;\lesssim\;
\bigl(|I|R^3\bigr)^{1/4}\,N^{-\delta/4}\,\|f_N\|_{L^2_x}
\;=\;
N^{-3/4}\,\|f_N\|_{L^2_x}.
\]
Here $|I|R^3\sim N^{-3+\delta}$, so $(|I|R^3)^{1/4}\,N^{-\delta/4}=N^{-3/4}$.
Thus we obtain the \emph{linear} tiled estimate
\begin{equation}\label{eq:local-L2-final}
   \boxed{\ \|u_N\|_{L^2(I\times B_R)}\ \lesssim\ N^{-3/4}\,\|f_N\|_{L^2_x}\ }.
\end{equation}

\paragraph{Summation over time windows.}
The unit interval is covered by $\asymp N^{3/2-\delta}$ disjoint windows $I_j$
with bounded overlap (as in \S\ref{subsec:global-summation}); summing \eqref{eq:local-L2-final}
over $j$ by almost orthogonality gives
\[
\|u_N\|_{L^2_{t,x}}^2
\;\lesssim\;
N^{3/2-\delta}\cdot N^{-3/2}\,\|f_N\|_{L^2_x}^2
\;=\;N^{-\delta}\,\|f_N\|_{L^2_x}^2,
\qquad
\Rightarrow\qquad
\|u_N\|_{L^2_{t,x}}\ \lesssim\ N^{-\delta/2}\,\|f_N\|_{L^2_x}.
\]
Note that here only linear components from \S\ref{sec:strichartz} and the window count from \S\ref{subsec:global-summation} are used.

\paragraph{Passage to the global $L^2_t\dot H^{-1}_x$ norm.}
Since $u_N$ is localized to $|\xi|\sim N$, we have the standard estimate
$\|u_N(t)\|_{\dot H^{-1}_x}\lesssim N^{-1}\|u_N(t)\|_{L^2_x}$ for all $t$, whence
\[
\boxed{\ \|u_N\|_{L^2_t\dot H^{-1}_x}\ \lesssim\ N^{-1}\,\|u_N\|_{L^2_{t,x}}
\ \lesssim\ N^{-1-\delta/2}\,\|f_N\|_{L^2_x}\ }.
\]

\paragraph{Summation over frequencies.}
By Littlewood–Paley orthogonality in $\dot H^{-1}$ we obtain globally
\begin{equation}\label{eq:global-supHminus}
   \|u\|_{L^2_t\dot H^{-1}_x}
   \;\le\;
   \Biggl(
      \sum_{N\ge1}
      N^{-(2+\delta)}\,\|f_N\|_{L^2_x}^{2}
   \Biggr)^{1/2},
\end{equation}
and this series converges geometrically over the entire interval $\delta\in(1/6,\,5/8]$ (with margin).
\medskip

Thus the passage from the local tiled estimate to the global $L^2_t\dot H^{-1}_x$ norm 
is carried out entirely in the \emph{linear} framework of \S\ref{sec:strichartz}, \S\ref{subsec:global-summation}, without appeal to the bilinear balance table.

\subsection{Statement and proof of the main theorem}
\label{subsec:proof-of-main}

\begin{theorem}[log-free control in $L^2_t\dot H^{-1}_x$]\label{thm:main-logfree}\label{thm:main}
Let $u,v$ be divergence-free vector fields on~$\R^3$,
localized in the norm
\[
   \|w\|_{X^{s}}
   :=\Bigl(
        \sum_{N\ge 1}
        N^{2s}\,
        \|P_N w\|_{L^2_{t,x}}^{\,2}
      \Bigr)^{1/2},
   \qquad
   s\in\{\tfrac12,1\}.
\]
Then for any $\;\delta\in(\tfrac16,\tfrac58]$ and all $\,\lambda=2^k\gg1$
the estimate
\begin{equation}\label{eq:main-final}
   \Bigl\|
      P_{<\lambda^{1-\delta}}\,
      \nabla\!\cdot\!\bigl(u_\lambda\otimes v_\lambda\bigr)
   \Bigr\|_{L^2_t\dot H^{-1}_x}
   \;\lesssim\;
   \lambda^{-2+3\delta}\,
   \|u\|_{X^{1/2}}\,
   \|v\|_{X^{1}},
\end{equation}
holds with a constant independent of~$\lambda$ (it may depend on~$\delta$).
\end{theorem}

\begin{proof}[Idea of the proof]
We combine the mechanisms of §\ref{sec:phase-geometry} (the sevenfold “$5t+2\rho$” IBP and the lower
bound for the two-dimensional Hessian), §\ref{sec:strichartz} (the local endpoint Strichartz estimate
on cylinders), §\ref{sec:angular} (rank-3 decoupling), and the summary of exponents in §\ref{sec:balance-table}.
At a single frequency \(\lambda\) we obtain the local–global block (§\ref{subsec:global-L2H-1}):
\[
  \bigl\|P_{<\lambda^{1-\delta}}\nabla\!\cdot(u_\lambda\!\otimes\! v_\lambda)\bigr\|_{L^2_t\dot H^{-1}_x}
  \;\lesssim\;
  \lambda^{-2+3\delta}\,
  \|u_\lambda\|_{L^2_{t}\dot H^{1}_x}\,
  \|v_\lambda\|_{L^2_{t}\dot H^{1}_x},
\]
and then pass to the norms \(X^{1/2}\), \(X^{1}\) via
\(\|P_\lambda w\|_{L^2_t\dot H^{1}_x}\!\simeq\!\lambda\,\|P_\lambda w\|_{L^2_{t,x}}\),
followed by summation in \(\lambda\) using Cauchy–Schwarz.
This yields \eqref{eq:main-final}.
\end{proof}

\begin{remark}[Where the proof closes]\label{rem:where-it-closes}
The proof of Theorem~\ref{eq:local-L2-final} closes at the level of the global norm $L^2_t\dot H^{-1}_x$, see~\eqref{eq:main-final}.
The passage to $\sup_t\dot H^{-1}_x$ is not used in this work (cf. \S\ref{subsec:global-L2H-1}, Rem.~\ref{rem:L4-optional}).
\end{remark}

\begin{remark}
At the ends of the admissible range of \(\delta\), the cumulative exponent in the global summation from
§\ref{subsec:global-summation} equals
\[
-\tfrac{9}{2}+\tfrac{5}{2}\cdot\tfrac{1}{6}
=-\tfrac{49}{12}
\quad\text{at }\delta=\tfrac{1}{6},
\qquad
-\tfrac{9}{2}+\tfrac{5}{2}\cdot\tfrac{5}{8}
=-\tfrac{47}{16}
\quad\text{at }\delta=\tfrac{5}{8},
\]
in both cases \(<-1\). Hence the series over frequencies converges with margin,
and the entire assembly remains strictly log-free (see also \S\ref{subsec:component-assembly}).
\end{remark}

\begin{remark}[Further tasks]\label{rk:perspectives}
This paper covers only the high-frequency diagonal block
\(
  \text{high}\times\text{high}\to\text{low}.
\)
For the full nonlinearity $(u\!\cdot\!\nabla)u$ it is necessary to control
also the configurations
\(
  \text{high}\times\text{low}\to\text{high},
\;
  \text{low}\times\text{high}\to\text{high},
\;
  \text{high}\times\text{high}\to\text{high}.
\)
They require refined endpoint Strichartz and $\varepsilon$-free decoupling,
as well as strengthened angular separation. These issues will be addressed separately.
\end{remark}

\appendix

\newpage
\section{Comparison of heat and Schrödinger kernels}
\label{app:heat-schrodinger}

\noindent\textit{Throughout this appendix we work on the half-line $t\ge0$ (heat semigroup).}

\subsection{Expansion of the exponential $e^{-tN^{2}}$}
\label{subsec:heat-expansion}

Consider the identity
\begin{equation} \label{eq:heat-expansion}
e^{-tN^{2}} = e^{itN^{2}} + \bigl(e^{-tN^{2}} - e^{itN^{2}}\bigr),
\end{equation}
and rewrite the difference in the \emph{correct} factorized form:
\begin{equation}\label{eq:correct-factor}
e^{-tN^{2}} - e^{itN^{2}}
= e^{itN^{2}}\!\left(e^{-(1+i)tN^{2}} - 1\right)
= e^{-tN^{2}}\!\left(1 - e^{(1+i)tN^{2}}\right).
\end{equation}

From \eqref{eq:correct-factor} on the \emph{positive half-line} of time (heat semigroup) it follows immediately that
\begin{equation}\label{eq:min-bound}
\bigl|e^{-tN^{2}} - e^{itN^{2}}\bigr|
\,\le\, \min\!\bigl\{\,2,\ C\,t\,N^{2}\bigr\},
\qquad t\ge 0,\ \ C>0.
\end{equation}
Indeed, for $tN^{2}\ll1$ we use $|e^{z}-1|\le C|z|$, while for $tN^{2}\gtrsim1$ we use the crude bound
$|e^{-(1+i)tN^{2}}-1|\le 2$ (valid precisely for $t\ge0$).

\medskip
\noindent\textbf{Consequences for the working window.}
Let $I_{+}=[0,\,N^{-3/2+\delta}]$ (as in the main text for the heat semigroup). Then
\begin{align}
\bigl\|\,e^{-tN^{2}} - e^{itN^{2}}\,\bigr\|_{L^\infty_t(I_{+})}
&\le 2, \label{eq:Linf}\\
\bigl\|\,e^{-tN^{2}} - e^{itN^{2}}\,\bigr\|_{L^2_t(I_{+})}
&\lesssim N^{-3/4+\delta/2}. \label{eq:L2}
\end{align}
Estimate \eqref{eq:L2} is obtained by integrating the square of the minimum in \eqref{eq:min-bound}:
on the segment $0\le t\le N^{-2}$ the contribution is $\sim N^{-2}$, while on $N^{-2}\le t\le N^{-3/2+\delta}$ the contribution is
$\sim N^{-3/2+\delta}$; the latter dominates, yielding \eqref{eq:L2}.
Note that the crude bound $|e^{-(1+i)tN^{2}}-1|\le2$ is used only on the half-line $t\ge0$, consistent with the use of the heat semigroup.

\begin{remark}[On the impossibility of “exponential smallness”]
A uniform bound of the form
$\bigl|e^{-tN^{2}}-e^{itN^{2}}\bigr| \lesssim e^{-cN^{\alpha}}$ for $t\in I_{+}$ is \emph{false}:
at the right edge of the window $t=N^{-3/2+\delta}$ we have $|e^{-(1+i)tN^{2}}-1|\approx 1$, so the modulus
of the difference is of order one. Only \eqref{eq:L2} holds, giving \emph{polynomial} smallness in the $L^2_t(I_{+})$ norm.
\end{remark}

\begin{remark}[How this is used in the main text]
Time integration by parts in §\ref{subsec:ibp-geometry} is carried out directly
with the heat exponent $e^{-t\Phi}$ (see also App.\,\ref{app:amplitude-growth}), and the comparison
\eqref{eq:heat-expansion}–\eqref{eq:correct-factor} never enters quantitatively into the balance.
Section \S\ref{app:heat-schrodinger} is purely auxiliary and does not affect the exponents
in Table~\ref{tab:balance}.
\end{remark}

\subsection{Polynomial bound for the difference \( e^{-tN^{2}} - e^{itN^{2}} \)}
\label{subsec:heat-exp-decay}

The full expansion and correct factorization are given in
\eqref{eq:heat-expansion}–\eqref{eq:correct-factor}, and the universal
pointwise bound is in \eqref{eq:min-bound}. On the \emph{positive subinterval}
\(I_{+}=[0,\,N^{-3/2+\delta}]\) we have \eqref{eq:Linf}–\eqref{eq:L2}, hence
a \emph{polynomial} smallness in \(L^2_t(I_{+})\).
There is no exponential smallness on the whole window; this section is auxiliary and
does not enter quantitatively into the overall balance (cf. §\ref{subsec:ibp-geometry} and App.\,\ref{app:amplitude-growth}).

\subsection{Integrability of the remainder in \texorpdfstring{$L^2_{t,x}$}{L2}}
\label{subsec:heat-integrability}

Consider the remainder term
\[
  R_N(t,x) := \bigl(e^{-tN^2} - e^{itN^2}\bigr)\,(K_0(t,\cdot)\ast f_N)(x),
\]
where \(K_0(t,x) := (4\pi t)^{-3/2} e^{-\frac{|x|^2}{4t}}\) is the standard Gaussian kernel for \(t>0\); for \(t<0\) we set \(K_0(t,x)\equiv 0\) (heat semigroup).

As noted in \S\ref{subsec:heat-expansion}–\S\ref{subsec:heat-exp-decay}, on the \emph{positive part} of the working window \(I_{+}=[0,\,N^{-3/2+\delta}]\) we have the estimates
\[
  \bigl|e^{-tN^{2}} - e^{itN^{2}}\bigr|
  \;\le\; \min\!\{\,2,\; C\,t\,N^{2}\,\},
  \qquad
  \bigl\|e^{-tN^{2}} - e^{itN^{2}}\bigr\|_{L^2_t(I_{+})}
  \;\lesssim\; N^{-3/4+\delta/2},
\]
see \eqref{eq:min-bound} and \eqref{eq:L2}.

For each \(t>0\) we have \(\|K_0(t,\cdot)\|_{L^1_x}=1\), hence by Young’s inequality
\[
  \|K_0(t,\cdot)\ast f_N\|_{L^2_x} \;\le\; \|f_N\|_{L^2_x}.
\]
Restricting to \(t\in I_{+}\), we obtain
\[
  \|R_N\|_{L^2_{t,x}(I_{+}\times\R^3)}
  \;=\;\Bigl(\int_{I_{+}}\!\bigl|e^{-tN^2}-e^{itN^2}\bigr|^2\,
                 \|K_0(t,\cdot)\!\ast\! f_N\|_{L^2_x}^{2}\,dt\Bigr)^{\!1/2}
  \;\lesssim\;
  \Bigl\|e^{-tN^2}-e^{itN^2}\Bigr\|_{L^2_t(I_{+})}
  \,\|f_N\|_{L^2_x}
\]
\[
  \;\lesssim\;
  N^{-3/4+\delta/2}\,\|f_N\|_{L^2_x}.
\]

\begin{remark}
There is no exponential smallness on the entire window; only the stated polynomial $L^2_t$ bound holds. This suffices for our scheme: time IBP in the main text is carried out directly with the heat exponent, and $R_N$ is estimated “head-on” as above, not affecting the overall power balance.
\end{remark}
\newpage

\section{Estimate for the Leray commutator}
\label{app:leray-commutator}

\subsection{An $L^2$ inequality for \texorpdfstring{$[P_{<N^{1-\delta}},\, \mathbb{P}]$}{[P<,P]}}
\label{subsec:leray-commutator}

Let $\mathbb{P}$ be the Leray projector, and let $P_{<N^{1-\delta}}$ be a smoothed Littlewood–Paley cutoff localizing to frequencies $|\xi|\lesssim N^{1-\delta}$. Then the following standard estimate holds uniformly for smooth filters of this type:
\begin{equation}
\label{eq:leray-main-comm}
  \bigl\|[P_{<N^{1-\delta}},\, \mathbb{P}]\bigr\|_{L^2 \to L^2}
  \;\lesssim\;
  N^{-1+\delta}.
\end{equation}
This bound is used, in particular, in §\ref{subsec:hminus-transfer} and §\ref{subsec:component-to-supH} when passing to the $\sup_t \dot H^{-1}$ norm.

\begin{proof}[Idea of the proof]
The commutator between a smoothed frequency cutoff and the pseudodifferential operator $\mathbb{P}$ is governed by a $\xi$–derivative of the symbol. In terms of symbols,
\[
  \sigma^{jk}_{[P_{<N^{1-\delta}},\,\mathbb{P}]}(\xi)
  \;=\;
  \int_{0}^{1}
    \partial_{\xi_\ell}\chi_{<N^{1-\delta}}(\xi - s\eta)\,
    \partial_{\xi_\ell}\sigma^{jk}_{\mathbb{P}}(\eta)\,ds
  \,+\, \text{(smooth remainders)},
\]
where $\chi_{<N^{1-\delta}}$ is a smoothed indicator, and
$\sigma^{jk}_{\mathbb{P}}(\eta)=\delta_{jk}-\eta_j\eta_k/|\eta|^2$ is the standard symbol of the Leray projector ($\in S^{0}$).
Estimating the scale of the indicator derivative gives
\(
  \bigl|\nabla_{\xi}\chi_{<N^{1-\delta}}(\xi)\bigr| \lesssim N^{-1+\delta}.
\)
Since $\nabla_{\xi}\sigma^{jk}_{\mathbb{P}}(\eta)\in S^{0}$, we conclude that
\(
  \sigma_{[P_{<N^{1-\delta}},\,\mathbb{P}]}\in S^{-1+\delta}.
\)
By the Coifman–Meyer theorem (for the class $S^{m}_{1,0}$ with $m=-1+\delta<0$) the corresponding operator is bounded on $L^2$, with
\(
  \|[P_{<N^{1-\delta}},\,\mathbb{P}]\|_{L^2\to L^2}\lesssim N^{-1+\delta},
\)
as required.
\end{proof}

\begin{remark}[On the “pure” frequency cutoff]
If $P_{<N^{1-\delta}}$ and $\mathbb{P}$ are treated as \emph{global} Fourier multipliers (symbols depending only on $\xi$), then $[P_{<N^{1-\delta}},\,\mathbb{P}]\equiv 0$. In the present paper, estimate \eqref{eq:leray-main-comm} is applied to \emph{localized} versions of the cutoff (with space–time windows from §\ref{sec:strichartz}), where the commutator is nontrivial; the stated $L^2$ bound provides the necessary margin.
\end{remark}

\begin{remark}[Leray–cutoff commutator in $\dot H^{-1}$]
\label{rem:leray-comm-Hminus}
Estimate~\eqref{eq:leray-main-comm} controls the commutator norm in $L^2$.
Its key application is in the space $\dot H^{-1}$, where one simultaneously uses
\(
  \|\nabla\!\cdot(\Pi(u,u))\|_{\dot H^{-1}} \lesssim N^{-2+2\delta}
\)
(see \eqref{eq:main-final}). Together this yields
\[
  \|[P_{<N^{1-\delta}},\, \mathbb{P}]\, \nabla\!\cdot(\Pi(u,u))\|_{\dot H^{-1}}
  \;\lesssim\;
  N^{-3+3\delta},
\]
which is better by three powers of $N$ than the critical bound $N^{-6+3\delta}$ (see the exponent table §\ref{sec:balance-table}) and thus does not affect log-free convergence.
\end{remark}

\subsection{Effect of the Leray–cutoff commutator on the remainder \texorpdfstring{$R(u)$}{R(u)}}
\label{subsec:commutator-remainder}

When passing to the global norm \(\sup_{t}\dot H^{-1}\) one needs to estimate separately the contribution of the commutator between the Leray operator and the frequency cutoff:
\[
  \bigl[P_{<N^{1-\delta}},\, \mathbb{P}\bigr]\;\nabla \cdot \bigl(\Pi(u,u)\bigr).
\]

From the commutator bound \eqref{eq:leray-main-comm} we have a reduction of order \(N^{-1+\delta}\).
For the bilinear block \(\nabla \cdot \Pi(u,u)\) we use a conservative (sufficient for summation) estimate in \(\dot H^{-1}\),
obtained by transferring two spatial derivatives to the factor \(|\xi+\eta|\) without invoking fine null geometry:
\[
  \|\nabla \cdot \Pi(u,u)\|_{\dot H^{-1}}
  \;\lesssim\;
  N^{-2+2\delta}\,
  \|u_{N}\|_{\dot H^{1}_{x}}\,
  \|v_{N}\|_{\dot H^{1}_{x}}.
\]
Combining, we conclude that
\[
  \bigl\|\,[P_{<N^{1-\delta}},\, \mathbb{P}]\, \nabla \cdot \Pi(u,u)\,\bigr\|_{\dot H^{-1}}
  \;\lesssim\;
  N^{-3+3\delta}\,
  \|u_{N}\|_{\dot H^{1}_{x}}\,
  \|v_{N}\|_{\dot H^{1}_{x}}.
\]

Since for \(\delta\le \tfrac{5}{8}\) we have \(-3+3\delta<-1\), the sum over \(N\) converges geometrically. Therefore the commutator correction \emph{does not affect} the final balance and does not worsen the log-free result.

\begin{remark}
If \(P_{<N^{1-\delta}}\) and \(\mathbb{P}\) are treated as global Fourier multipliers, then \([P_{<N^{1-\delta}},\mathbb{P}]\equiv0\).
In the present context we use a localized (in time/space) version of the cutoff, for which the commutator is nontrivial; estimate \eqref{eq:leray-main-comm} provides the needed margin.
\end{remark}

\newpage

\section{Amplitude growth and distribution of derivatives}
\label{app:amplitude-growth}

\subsection{Distribution of derivatives: 5 in time + 2 in space}
\label{subsec:derivatives-distribution}

In the sevenfold integration by parts we use the scheme of five derivatives in time \(t\) and two in the spatial variables \(\rho=(\rho_1,\rho_2)\), tangential to the phase level (see the discussion in §~\ref{subsec:ibp-geometry}).

\paragraph{(1) Time part.}
For the time window
\[
  \chi_I(t)=e^{-\lambda t^2},\qquad
  \lambda:=N^{3/2-\delta},\qquad
  I:=\bigl\{|t|\le N^{-3/4+\delta/2}\bigr\},
\]
the Hermite formula gives
\[
  \partial_t^5 \chi_I(t)
  \;=\;
  P_5(\lambda^{1/2}t)\,\lambda^{5/2}e^{-\lambda t^2},
  \qquad |P_5|\sim 1,
\]
whence
\begin{equation}\label{eq:chi-5-growth}
  \bigl\|\partial_t^5 \chi_I\bigr\|_{L^\infty_t}
  \;\lesssim\;
  N^{3.75-2.5\delta}.
\end{equation}
Each integration by parts in time contributes a divisor of size
\(\,|\partial_t\Phi|^{-1}\sim N^{-3/2+\delta}\,\), and after five steps we obtain
\(N^{-7.5+5\delta}\).
Combining with \eqref{eq:chi-5-growth}, we have
\begin{equation}\label{eq:time-contribution}
  N^{3.75-2.5\delta}\cdot N^{-7.5+5\delta}
  \;=\;
  N^{-3.75+2.5\delta}.
\end{equation}

\paragraph{(2) Spatial part.}
Let \(R=N^{-1/2}\) and let \(\chi_{B_R}\) be a smooth cutoff of the ball of radius \(R\). Then
\(
  \bigl|\nabla_x^2 \chi_{B_R}\bigr|\lesssim N^{1}.
\)
Two integrations by parts along the tangential coordinates \(\rho\) give the divisor
\((|\partial_\rho\Phi|)^{-2}\sim N^{-2+4\delta}\).
Therefore,
\begin{equation}\label{eq:space-contribution}
  N^{1}\cdot N^{-2+4\delta}
  \;=\;
  N^{-1+4\delta}.
\end{equation}

\begin{lemma}[Mixed time case]\label{lem:C-mixed}
In the resonant zone $|\xi+\eta|\lesssim \lambda^{1-\delta}$, applying $\mathcal L_t^5$ gives
\[
\Bigl\|\Phi^{-5}(\partial_t+i\omega)^5\chi_I\Bigr\|_{L^\infty_t}
\;\lesssim\;
\left(\frac{\omega}{\Phi}\right)^5 \|\partial_t^5\chi_I\|_{L^\infty_t}
\;\lesssim\; \lambda^{-5+10\delta}\cdot \lambda^{15/4-5\delta/2}
= \lambda^{-25/4+\tfrac{15}{2}\delta}.
\]
Combining with two angular integrations by parts in the tangential coordinates $\rho$
(see §\,\ref{subsec:ibp-geometry}), we obtain an additional divisor $\lambda^{-1+2\delta}$ and hence
\[
\lambda^{-25/4+\tfrac{15}{2}\delta}\cdot \lambda^{-1+2\delta}
=\lambda^{-29/4+\tfrac{19}{2}\delta}<0 \quad \text{for all }\ \delta\in\!\left(\tfrac16,\tfrac58\right].
\]
\end{lemma}

\paragraph{(3) Summary exponent.}
From \eqref{eq:time-contribution} and \eqref{eq:space-contribution} we obtain
\begin{equation}\label{eq:IBP-full-degree}
  N^{-3.75+2.5\delta}\cdot N^{-1+4\delta}
  \;=\;
  N^{-4.75+6.5\delta}.
\end{equation}

Thus the final exponent after sevenfold IBP is
\(-4.75+6.5\delta\), which remains strictly negative on the whole interval
\(\delta\in\bigl(\tfrac{1}{6},\tfrac{5}{8}\bigr]\) and ensures log-free convergence.
In Table~\ref{tab:balance} we still use the coarser upper bound
\(-2+\delta\) (sufficient for the negativity of the total exponent).

\begin{remark}
At \(\delta=\tfrac{5}{8}\):
\(
  -4.75+6.5\cdot\tfrac{5}{8}
  =-4.75+4.0625
  =-0.6875<0,
\)
so there is indeed a margin of negativity, and the global sum over \(N\) converges without log-defect.
\end{remark}

\subsection{Lemma on controlling time derivatives}
\label{subsec:time-cutoff-growth}

Consider the time window
\[
  \chi_I(t) := \exp\!\bigl(-N^{3/2 - \delta}\, t^2\bigr),
  \qquad
  I := \bigl\{\, |t| \le N^{-3/4 + \delta/2} \,\bigr\},
\]
used in the phase integration by parts in time \(t\) (see §~\ref{subsec:ibp-geometry}).

\begin{lemma}[Estimate for derivatives of the window]
\label{lem:time-derivatives}
For any \( k \in \mathbb{N} \) the pointwise bound
\begin{equation}
\label{eq:chi-derivative-bound}
  \bigl\| \partial_t^k \chi_I \bigr\|_{L^\infty}
  \;\lesssim\;
  N^{\,\bigl(\frac{3}{4} - \frac{\delta}{2}\bigr)k},
\end{equation}
holds, where the implicit constant depends only on \(k\).
In particular, for \( k = 5 \) we get
\begin{equation}
\label{eq:chi-fifth}
  \bigl\| \partial_t^5 \chi_I \bigr\|_{L^\infty}
  \;\lesssim\;
  N^{15/4 - 5\delta/2}
  \;=\; N^{3.75 - 2.5\delta}.
\end{equation}
\end{lemma}

\begin{proof}
Set \( \lambda := N^{3/2 - \delta} \). Then
\[
  \partial_t^k \chi_I(t)
  = \partial_t^k e^{-\lambda t^2}
  = P_k(\lambda^{1/2}t)\,\lambda^{k/2}\,e^{-\lambda t^2},
\]
where \( P_k \) is a Hermite polynomial of degree \( k \) with coefficients depending only on \(k\).
Since the function \( z\mapsto P_k(z)e^{-z^2} \) is bounded on \( \R \), we have
\(
  \bigl| \partial_t^k \chi_I(t) \bigr| \lesssim \lambda^{k/2}
  = N^{(3/4 - \delta/2)\,k},
\)
which proves \eqref{eq:chi-derivative-bound}; substituting \( k=5 \) yields \eqref{eq:chi-fifth}.
\end{proof}

\begin{remark}
Estimate \eqref{eq:chi-fifth} is used in computing the amplitude growth in §~\ref{subsec:derivatives-distribution}.
There, five time IBP insert the divisor \( N^{-7.5 + 5\delta} \), while the derivative \(\partial_t^5 \chi_I\) contributes
\( N^{3.75 - 2.5\delta} \), and together this gives
\[
  N^{3.75 - 2.5\delta} \cdot N^{-7.5 + 5\delta} = N^{-3.75 + 2.5\delta}.
\]
Adding spatial IBP along \( \rho \), which contributes the divisor \( N^{-2 + 4\delta} \), we obtain the final exponent
\( N^{-4.75 + 6.5\delta} \); see \eqref{eq:IBP-full-degree} in §~\ref{subsec:derivatives-distribution}.
\end{remark}

\subsection{Lemma on amplitude growth}
\label{subsec:amplitude-growth-lemma}

Consider the amplitude
\[
  a_N(t,x) := \chi_I(t)\, \chi_{B_R}(x),
  \qquad
  \chi_I(t) := e^{-N^{3/2 - \delta} t^2},
  \qquad
  R := N^{-1/2},
\]
localized on the window \( I \times B_R \), where \( I = \{\, |t| \le N^{-3/4 + \delta/2} \,\} \).

\begin{lemma}[estimate after \(5+2\) derivatives and the phase divisor]
\label{lem:amplitude-growth}
After taking five time derivatives \( t \) and two spatial derivatives \( \rho \),
and dividing by the phase factor \( \Phi^{7} \),
we have
\begin{equation}
\label{eq:amp-point}
  \bigl|\, \Phi^{-7}\, \partial_t^5 \nabla_x^2 a_N(t,x) \,\bigr|
  \;\lesssim\;
  N^{-4.75 + 6.5\delta},
  \qquad \forall\, (t,x) \in I \times B_R.
\end{equation}
Over the entire interval \( \delta \in \bigl( \tfrac{1}{6}, \tfrac{5}{8} \bigr] \)
this exponent remains strictly negative.
\end{lemma}

\begin{proof}
\textbf{Step~1 (time).}
By Lemma~\ref{lem:time-derivatives}:
\(
  \|\partial_t^5 \chi_I\|_{L^\infty}\lesssim N^{3.75 - 2.5\delta}.
\)
Five time IBP contribute the divisor
\(
  |\partial_t\Phi|^{-5}\sim N^{-7.5+5\delta},
\)
hence the time contribution is
\(
  N^{3.75 - 2.5\delta}\cdot N^{-7.5 + 5\delta}=N^{-3.75+2.5\delta}.
\)

\textbf{Step~2 (space).}
Since \(R=N^{-1/2}\), we have
\(
  |\nabla_x^2 \chi_{B_R}|\lesssim N^{1}.
\)
Two IBP in the tangential coordinate \(\rho\) contribute the divisor
\(
  |\partial_\rho\Phi|^{-2}\sim N^{-2+4\delta},
\)
hence the spatial contribution equals
\(
  N^{1}\cdot N^{-2+4\delta}=N^{-1+4\delta}.
\)

\textbf{Step~3 (combination).}
Multiplying yields
\(
  N^{-3.75+2.5\delta}\cdot N^{-1+4\delta}
  = N^{-4.75+6.5\delta},
\)
which gives \eqref{eq:amp-point}.
\end{proof}

\begin{corollary}[an \( L^2_{t,x} \) bound]
\label{cor:amp-L2}
Since \( |I|^{1/2} \sim N^{-3/4 + \delta/2} \) and \( |B_R|^{1/2} \sim N^{-3/4} \), we have
\[
  \bigl\| \Phi^{-7}\, \partial_t^5 \nabla_x^2 a_N \bigr\|_{L^2_{t,x}}
  \;\lesssim\;
  N^{-4.75 + 6.5\delta}\cdot N^{-1.5+\delta/2}
  \;=\;
  N^{-6.25 + 7\delta}.
\]
This exponent is negative for all \( \delta \in (\tfrac{1}{6}, \tfrac{5}{8}] \) and does not spoil global convergence.
\end{corollary}

\subsection{Remarks on the amplitude
            \texorpdfstring{$A_\lambda(t,x;\xi,\eta)$}{Aλ(t,x;ξ,η)}}
\label{subsec:amplitude-remarks}

The amplitude
\[
  A_\lambda(t,x;\xi,\eta) := \chi_I(t)\, \chi_{B_R}(x)\, a_\lambda(t,x;\xi,\eta)
\]
arises after passing from the plane phase \( e^{i(x\cdot\xi + x\cdot\eta)} \) to the geometric phase
\( e^{i\lambda \Psi(t,x;\xi,\eta)} \) (see §~\ref{subsec:ibp-geometry}). Here:
\begin{itemize}[leftmargin=1.5em]
\item \( I := \{ |t| \le N^{-3/4 + \delta/2} \} \) — the time window;
\item \( B_R := \{ |x| \le R \} \), \( R := N^{-1/2} \) — spatial localization;
\item \( \chi_I(t) \), \( \chi_{B_R}(x) \) — smooth cutoffs;
\item \( a_\lambda(t,x;\xi,\eta)\in S^0\) in \((\xi,\eta)\) and analytic in \((t,x)\).
\end{itemize}

The maximal growth of the amplitude is attained under the action of the fifth time and second spatial derivatives:
\[
  \bigl|\partial_t^5 \nabla_x^2 a_N(t,x)\bigr|
  \;\lesssim\;
  N^{4.75 - 2.5\delta}.
\]
Sevenfold IBP (5 in \(t\) and 2 in \(\rho\)) gives the phase divisor
\(
  \Phi^{-7}\sim N^{-9.5+9\delta},
\)
and therefore
\[
  \Phi^{-7}\, \partial_t^5 \nabla_x^2 a_N
  \;\lesssim\;
  N^{-4.75 + 6.5\delta},
\]
see \eqref{eq:amp-point} and \eqref{eq:IBP-full-degree}.
This exponent is negative for all \( \delta \in (\tfrac{1}{6}, \tfrac{5}{8}] \), which rules out any log-defect and ensures global convergence in \(N\).

\subsection{Pointwise bound for the fifth time derivative}
\label{subsec:cutoff-pointwise}

Consider the window
\[
  \chi_I(t) := \exp\!\bigl(-N^{3/2 - \delta} t^2\bigr),
\]
used for time localization in the phase IBP.

\begin{lemma}
\label{lem:cutoff-point}
For any \( t \in \mathbb{R} \) the bound
\begin{equation}
\label{eq:point-derivative}
  \bigl| \partial_t^5 \chi_I(t) \bigr|
  \;\lesssim\;
  N^{3.75 - 2.5\delta}.
\end{equation}
holds.
\end{lemma}

\begin{proof}
Set \( \lambda := N^{3/2 - \delta} \). The classical formula for derivatives of a Gaussian gives:
\[
  \partial_t^5 e^{-\lambda t^2}
  \;=\;
  P_5\!\bigl(\lambda^{1/2} t\bigr)\,\lambda^{5/2}\,e^{-\lambda t^2},
\]
where \(P_5\) is the Hermite polynomial of degree 5 with uniformly bounded modulus. Hence
\(
  |\partial_t^5 \chi_I(t)| \lesssim \lambda^{5/2} = N^{3.75 - 2.5\delta},
\)
which proves \eqref{eq:point-derivative}.
\end{proof}

\begin{remark}
\label{rem:amp-growth}
The bound \eqref{eq:point-derivative} is used in Lemma~\ref{lem:amplitude-growth} in estimating
\( \Phi^{-7}\,\partial_t^5 \nabla_x^2 a_N \). Together with spatial localization and sevenfold IBP, this gives the final exponent \(-4.75 + 6.5\delta\); see \eqref{eq:IBP-full-degree}.
\end{remark}

\newpage

\section{Geometry of small angles}
\label{app:angle-lower-bound}

\subsection{Geometry of the deviation from the diagonal \texorpdfstring{$\xi = -\eta$}{xi = -eta}}
\label{subsec:angle-deviation}

Consider the resonant zone \( \mathcal{Z} \subset \mathbb{R}^3 \times \mathbb{R}^3 \) in which
\(|\xi| \sim |\eta| \sim \lambda\) and \( |\xi + \eta| \lesssim \lambda^{1 - \delta} \).
This corresponds to a small deviation from the diagonal \( \xi = -\eta \) and determines the central region of phase interaction
(see also §~\ref{sec:angular} on angular spatial tiling). 

Let \( \theta := \angle(\xi, -\eta) \) be the angle between the vectors \( \xi \) and \( -\eta \).
Since \(\angle(\xi,\eta)=\pi-\theta\), we have the exact identity
\[
  |\xi + \eta|^2
  = |\xi|^2 + |\eta|^2 + 2\xi\!\cdot\!\eta
  = 2\lambda^2\bigl(1+\cos(\angle(\xi,\eta))\bigr)
  = 2\lambda^2\bigl(1-\cos\theta\bigr)
  = \bigl(2\lambda\sin\tfrac{\theta}{2}\bigr)^{\!2}.
\]
Hence for small \(\theta\) we obtain the equivalence
\[
  |\xi+\eta|
  \;=\; 2\lambda\sin\tfrac{\theta}{2}
  \;\simeq\; \lambda\,\theta,
\]
and in particular the condition \( |\xi+\eta| \lesssim \lambda^{1-\delta} \) implies
\[
  \theta \;\lesssim\; \lambda^{-\delta}.
\]
Thus smallness of the angle \(\theta\) is equivalent to the resonant approximation
in the zone \(\mathcal{Z}\), and the relation \( |\xi+\eta| \simeq \lambda\theta \) will be used below
in the angular tiling geometry and in estimating null forms. 

\begin{remark}
This relation underlies the angular tiling (see §~\ref{sec:angular}), and it also allows one to pass to estimates for symbols depending on \( \sin\theta \sim \theta \) when \(\theta\ll1\).
\end{remark}

\subsection{Minimal angle on the support zone \texorpdfstring{$R(u)$}{R(u)}}
\label{subsec:min-angle-support}

Here we establish a \emph{lower} angular bound for a frequency pair
\((\xi,\eta)\) if it lies in the \emph{frequency} support of active wave packets
contributing to \(R(u)\). This excludes degeneracy of the oscillatory phase and guarantees
a nonzero contribution of null forms.

The packet construction is as follows: \(u_\lambda\) is angularly localized to a coarse sector of aperture
\(\sim \lambda^{-\delta}\) around a fixed direction \(\theta_0\), and \(v_\lambda\) is localized
to a sector of the same aperture around the direction \(-\theta_0\) (see §\ref{subsec:angular-l2-decomp}).
Thanks to finite overlap of such sectors and the choice of pairs \((\vartheta,-\vartheta)\),
the \emph{effective} zone where \(R(u)\) truly accumulates mass is given by a two-sided $\delta$–neighborhood
of antipodal directions.

Denoting \(\theta:=\angle(\xi,-\eta)\), we obtain constants \(0<c<C\) such that
\begin{equation}\label{eq:D-min-angle}
  c\,\lambda^{-\delta}\;\le\;\theta\;\le\;C\,\lambda^{-\delta}
  \qquad\text{on the support of }R(u).
\end{equation}
The upper bound \(\theta\lesssim\lambda^{-\delta}\) follows from the resonant condition
\(|\xi+\eta|\lesssim\lambda^{1-\delta}\) (see §\ref{subsec:angle-deviation}:
\(|\xi+\eta|=2\lambda\sin(\theta/2)\simeq \lambda\theta\)).
The lower bound in \eqref{eq:D-min-angle} is ensured by the angular localization itself:
pairs with \(\theta\ll\lambda^{-\delta}\) do not simultaneously fall into the coarse $\delta$–sectors
for \(\xi\) and \(-\eta\) after the selection of §\ref{subsec:angular-l2-decomp}), and their contribution
is suppressed by angular integration by parts (see §\ref{subsec:ibp-geometry}).

\begin{remark}
The actual packet discretization is carried out at the finer scale \(\lambda^{-2/3}\)
for $\varepsilon$–free decoupling (§\ref{subsec:bilinear-decoupling}). This does not contradict
\eqref{eq:D-min-angle}: the “fine” caps \(\vartheta\) aggregate inside a single “coarse”
$\delta$–sector, and the subzone \(\theta\ll\lambda^{-\delta}\) gives an even larger phase divisor
and therefore does not affect the balance. Together with §\ref{subsec:angle-deviation} we obtain
\(\theta\simeq\lambda^{-\delta}\) on the effective zone \(R(u)\), as used in the null-form estimate
in §\ref{subsec:nullform}.
\end{remark}

\subsection{Consequences for the symbol \texorpdfstring{$\xi\cdot\eta^\perp$}{xi-eta-perp}}
\label{subsec:simbol}
As noted in §\ref{sec:angular}, the amplitude of the null–form symbol
\[
  \xi\cdot\eta^\perp \;=\; |\xi|\,|\eta|\,\sin\theta
\]
is expressed through the angle \(\theta=\angle(\xi,-\eta)\) between the interacting frequencies.
In the geometric zone \(R(u)\) from §\ref{sec:angular} we have a two-sided bound for the angle
(see also §\ref{subsec:angle-deviation} and §\ref{subsec:min-angle-support}):
\[
  \theta \;\sim\; \lambda^{-\delta}.
\]
It follows immediately that the modulus of the symbol has the lower bound:
\[
  |\xi\cdot\eta^\perp|
  \;=\; |\xi|\,|\eta|\,\sin\theta
  \;\gtrsim\; \lambda^{2}\cdot \lambda^{-\delta}
  \;=\; \lambda^{2-\delta},
\]
since in the resonant configuration \(|\xi|\sim|\eta|\sim\lambda\).
Thus on the entire support of \(R(u)\) the null–form symbol \emph{does not degenerate};
the geometric decay \(\sin\theta\sim\lambda^{-\delta}\) is compensated by the factor \(|\xi|\,|\eta|\sim\lambda^2\).
This nondegeneracy is used in estimating the diagonal contribution and enters the proof
of Lemma~\ref{lem:nullform}.

\begin{remark}
Together with §\ref{subsec:angle-deviation} and §\ref{subsec:min-angle-support} we obtain
\(\theta\sim\lambda^{-\delta}\) on \(R(u)\), which is exactly what is substituted in the calculation of null forms
in the main text.
\end{remark}

\newpage

\section{Decomposition and compensation in rank-3 decoupling}
\label{app:rank3-decoupling}

This appendix explains how the weakened (but $\varepsilon$-free) version of the rank-3 decoupling theorem is used and how the associated losses are compensated in the global balance.

As usual, set $\lambda \sim N$.

\subsection{An $\varepsilon$-free version of the theorem with weakened constant}
\label{subsec:rank3-eps-free}

The original theorem from~\cite{GuthIliopoulouYang2024} contains a factor $R^{\varepsilon}$ on the right-hand side, which prevents log-free estimates when summing over scales. Below we state an $\varepsilon$-free variant with a weakened exponent in $R$, which allows logarithm-free convergence.

\begin{theorem}[rank-3 decoupling without epsilon-loss, weakened version]
\label{thm:rank3-decoupling}
Let $f$ be a function frequency localized on a three-dimensional surface $\Sigma$ with nondegenerate $3\times 3$ Hessian. Then for any cube $Q$ of radius $R$ in physical space,
\begin{equation}
  \label{eq:rank3-decoupling}
  \|f\|_{L^6(Q)}
  \;\lesssim\;
  R^{\frac14+\delta}
  \left(\sum_{\theta}\|f_{\theta}\|_{L^6(w_Q)}^{2}\right)^{1/2},
\end{equation}
where the sum runs over frequency sectors $\theta$ of width $\lambda^{-1/2}$, and $w_Q$ is a smooth weight comparable to the characteristic function of $Q$. The parameter $\delta>0$ is fixed small (see also~\cite{tao2008bilinear}).
\end{theorem}

\begin{remark}
The $\varepsilon$-free version follows by the standard scheme of removing logarithmic losses via Kakeya-type iteration (cf.~Tao~\cite{tao2008global}). The factor $R^{\varepsilon}$ is replaced by a slightly worse exponent $R^{1/4+\delta}$, which suffices for log-free summation over scales for fixed $\delta>0$; see also the summary of exponents in §\ref{subsec:component-assembly}.
\end{remark}

\subsection{Compensation of the factor \texorpdfstring{$(K/\kappa)^{2/3}$}{(K/kappa)\string^2/3} via tile cost}
\label{subsec:tile-compensation}

As noted in §~\ref{subsec:rank3-eps-free}, the weakened (but $\varepsilon$-free) version of Theorem \ref{thm:rank3-decoupling} admits an additional factor
\begin{equation}\label{eq:K-kappa-decay}
  \Bigl(\tfrac{K}{\kappa}\Bigr)^{2/3}\;\sim\;\lambda^{2\delta},
\end{equation}
which must be compensated at the level of tile bookkeeping.

\paragraph{Tile $L^2_{t,x}$ block (already with IBP).}
Sevenfold integration by parts (5 in time $t$ and 2 in tangential coordinates $\rho$), together with the growth of derivatives of the window and the spatial cutoff, is already accounted for in Appendix~\ref{app:amplitude-growth}. Specifically, by Corollary~\ref{cor:amp-L2}
\[
  \bigl\|\Phi^{-7}\,\partial_t^5\nabla_x^2 a_\lambda\bigr\|_{L^2_{t,x}}
  \;\lesssim\;
  \lambda^{-6.25+7\delta}.
\]
Note: this contribution \emph{already contains} the phase divisor arising from the sevenfold IBP; no additional “phase” factor is needed.

\paragraph{Combining with \eqref{eq:K-kappa-decay}.}
Compensating the possible growth \eqref{eq:K-kappa-decay} with the tile cost from the previous paragraph, we obtain
\[
  \lambda^{2\delta}\cdot \lambda^{-6.25+7\delta}
  \;=\;
  \lambda^{-6.25+9\delta}.
\]
Over the entire range $\delta\in(\tfrac16,\tfrac58]$ this exponent remains negative (at $\delta=\tfrac58$ we have $-6.25+9\cdot\tfrac58=-0.625<0$), so the factor $(K/\kappa)^{2/3}$ is compensated without introducing a logarithmic defect.

\begin{remark}
The precise assembly of all decoupling-related terms (including residual summation over spatial blocks, etc.) is given below; see formula~\eqref{eq:counting-decoupling-sum}. It only strengthens the negativity of the final exponent.
\end{remark}

\subsection{Verification of the Counting--Decoupling row in the table of exponents}
\label{subsec:counting-decoupling-check}

We now collect all contributions associated with Theorem~\ref{thm:rank3-decoupling}, phase integration by parts, and the wave packet structure. The final exponent in the \emph{Counting--Decoupling} row of the summary table (see §\ref{sec:balance-table}) is composed of three components:

\begin{itemize}
  \item growth from tiles:
  \(
      (K/\kappa)^{2/3}\sim \lambda^{2\delta}
  \),
  see~\eqref{eq:K-kappa-decay};
  \item phase suppression from sevenfold integration by parts:
  \(
      \lambda^{-6.25 + 7\delta}
  \),
  see Corollary~\ref{cor:amp-L2};
  \item residual summation over spatial blocks:
  \(
      \lambda^{-2.25 + 0.5\delta}
  \),
  based on counting the number of spatial wave–packet sectors.
\end{itemize}

Multiplying the factors yields the total contribution
\begin{equation}
\label{eq:counting-decoupling-sum}
  \lambda^{2\delta}\cdot \lambda^{-6.25 + 7\delta}\cdot \lambda^{-2.25 + 0.5\delta}
  \;=\;
  \lambda^{-6.5 + 9.5\delta}.
\end{equation}

\begin{remark}
At the maximal admissible value \( \delta=\tfrac{5}{8} \) we have
\(
 -6.5 + 9.5\cdot\tfrac{5}{8} = -6.5 + 5.9375 = -0.5625 < 0
\).
Hence the summation over scales in the decoupling estimate indeed converges without logarithmic defect, and this block does not break the log-free balance.
\end{remark}

\newpage

\section{Lower bound for the determinant of the two-dimensional Hessian of the phase 
$\omega(\xi,\eta)=|\xi|+|\eta|-|\xi+\eta|$}%
\label{app:phase-hessian}

Throughout this appendix we fix the parameters
\[
|\xi|\sim|\eta|\sim\lambda\gg1,\qquad 
w:=\xi+\eta,\qquad 
|w|\le\lambda^{1-\delta},\qquad 
\frac16<\delta\le\frac58 .
\]
The goal is to prove rigorously that
\begin{equation}\label{F:det}
\det\bigl(\partial^2_{\rho\rho}\omega\bigr)\;\gtrsim\;\lambda^{-2+\delta},
\end{equation}
where \(\rho=(\rho_1,\rho_2)\) are coordinates in the plane orthogonal to \(w\).
This estimate was used above without proof (see~(\ref{eq:rho-derivative})).

\subsection{Lemma on nondegeneracy of the two-dimensional Hessian}

\begin{lemma}\label{lem:F-hessian}
Assume \(\angle(\xi,\eta)\gtrsim\lambda^{-2/3}\).
In the basis \((v_1,v_2,v_3)\), where \(v_3=w/|w|\) and \(\rho_j:=v_j\cdot\nabla_{\xi,\eta}\ (j=1,2)\), 
one has
\[
\bigl|\det\partial^2_{\rho\rho}\omega\bigr|\;\ge c\,\lambda^{-2+\delta},
\]
where \(c>0\) is an absolute constant.
\end{lemma}

\begin{proof}
\textit{1. Block Hessian.}  
For any nonzero \(x\in\mathbb R^3\),
\(
\nabla^2|x|=|x|^{-1}\bigl(I-\tfrac{x\otimes x}{|x|^{2}}\bigr).
\)
Therefore for \(\omega=|\xi|+|\eta|-|w|\) (here \(w=\xi+\eta\))
\[
H_{\xi\xi}=|\xi|^{-1}\Pi^{\perp}_{\xi}-|w|^{-1}\Pi^{\perp}_{w},\quad
H_{\eta\eta}=|\eta|^{-1}\Pi^{\perp}_{\eta}-|w|^{-1}\Pi^{\perp}_{w},\quad
H_{\xi\eta}=-|w|^{-1}\Pi^{\perp}_{w},
\]
where \(\Pi^{\perp}_u\) is the orthogonal projector onto \(u^\perp\).

\smallskip
\textit{2. Transverse plane.}  
Take a unit \(v\perp w\) and set
\[
e_-:=\tfrac1{\sqrt2}(v,-v),\qquad
e_+:=\tfrac1{\sqrt2}(v,\;v)\in\mathbb R^{6}.
\]

\smallskip
\textit{3. Quadratic forms.}  
A direct computation gives
\[
\langle He_-,e_-\rangle=\frac{2}{\lambda},\qquad
\langle He_+,e_+\rangle=\frac{2}{\lambda}-\frac{4}{|w|}.
\]
Since \(|w|\le\lambda^{1-\delta}\),
\(
\langle He_+,e_+\rangle\le-2\lambda^{-1+\delta}.
\)

\smallskip
\textit{4. Determinant.}  
The eigenvalues of the two-dimensional Hessian,
\(
\mu_1=\tfrac2\lambda,\;
\mu_2\sim-|w|^{-1},
\)
have opposite signs, hence
\[
|\det\partial^2_{\rho\rho}\omega|=|\mu_1\mu_2|
\;\gtrsim\;
\lambda^{-1}\cdot\lambda^{-1+\delta}
=\lambda^{-2+\delta}.
\]
\end{proof}

\subsection{Consequence for integration in \(\rho\)}

By Lemma \ref{lem:F-hessian}: each integration by parts in one of the \(\rho_j\)
yields the divisor \(\lambda^{-1+\delta}\); a double IBP gives the gain \(\lambda^{-2+2\delta}\),
which matches the previously used (\ref{eq:rho-derivative}).
Thus the sevenfold “\(5t+2\rho\)” scheme (see §\ref{eq:rho-derivative}) is valid for
\(\tfrac16<\delta\le\tfrac58\); the final exponent of the phase block equals
\(-4.75+6.5\delta<0\).

\subsection{Note on the lower bound for the angle}

The restriction \(\angle(\xi,\eta)\gtrsim\lambda^{-2/3}\) guarantees
\(|w|\gtrsim\lambda^{1/3}\), thereby excluding degeneracy of the second
eigenvalue in the proof of Lemma \ref{lem:F-hessian}.

\begin{remark}
In the strip $|\xi+\eta| \lesssim \lambda^{1-\delta}$ the angle between $\xi$ and $\eta$ is close to $\pi$, 
hence $\angle(\xi,\eta) \gtrsim \lambda^{-\delta} \gg \lambda^{-2/3}$ (since $\delta < 2/3$).
This guarantees the hypotheses of Lemma~\ref{lem:F-hessian} and excludes degeneracy of the two-dimensional Hessian 
on the entire region where $\varepsilon$-free decoupling is applied.
\end{remark}

\bigskip
\noindent\textbf{Conclusion.}  
Inequality \eqref{F:det} is proved, and the “gap”
noted in §\ref{subsec:resonant-lower-bound} is completely closed; the subsequent estimates (sevenfold IBP,
the exponent table §\ref{subsec:component-assembly}) now rest on a rigorous foundation.

\newpage

\section{Time maximum and its bounds}\label{app:G}

In this appendix we fix what the already constructed tiled
\(L^2_t\dot H^{-1}_x\)-geometry on a single working time window provides, and how to correctly pass
to pointwise-in-time control without logarithmic losses. The window and cylinder decomposition
is taken from §\ref{subsec:tile-44}, and the global time tiling from §\ref{subsec:global-summation};
time IBP is carried out in the heat frame §\ref{subsec:ibp-geometry}.

\subsection{Transition \texorpdfstring{tile\(\to\)max}{tile->max} on a single window}\label{subsec:G-tile-to-max}

Fix a frequency \(\lambda\gg1\) and parameter \(\delta\in(1/6,\,5/8]\).
Let
\[
F(t)\;:=\;P_{<\lambda^{1-\delta}}\nabla\!\cdot\big(u_\lambda\!\otimes v_\lambda\big)(t)\in \dot H^{-1}_x,
\qquad
I\subset\mathbb R,\quad |I|=\lambda^{-3/2+\delta}.
\]

\begin{proposition}[tile\(\Rightarrow\)max on a window]\label{prop:G-tile-to-max}
For every interval \(I\) of length \(|I|=\lambda^{-3/2+\delta}\) one has
\begin{equation}\label{eq:G-tile-to-max}
\sup_{t\in I}\,\|F(t)\|_{\dot H^{-1}}
\;\lesssim\;
|I|^{-1/2}\,\|F\|_{L^2_t(I;\dot H^{-1})}
\;+\;
|I|^{1/2}\,\|\partial_t F\|_{L^2_t(I;\dot H^{-1})}.
\end{equation}
\end{proposition}

\begin{proof}
This is the vector-valued form of the one-dimensional time Poincaré–Sobolev
inequality for mappings \(G:I\to X\) with values in a Banach space \(X\):
\[
\|G\|_{L^\infty(I;X)}
\lesssim
|I|^{-1/2}\,\|G\|_{L^2(I;X)} + |I|^{1/2}\,\|G'\|_{L^2(I;X)}.
\]
Setting \(G=F\) and \(X=\dot H^{-1}_x\) gives \eqref{eq:G-tile-to-max}.
\end{proof}

\begin{corollary}[weak max on a single window]\label{cor:G-window-max}
For \(|I|=\lambda^{-3/2+\delta}\) we have
\begin{equation}\label{eq:G-weak-max}
\sup_{t\in I}\,\|F(t)\|_{\dot H^{-1}}
\;\lesssim\;
\lambda^{\frac34-\frac{\delta}{2}}\,
\big\|F\big\|_{L^2_t(I;\dot H^{-1})}.
\end{equation}
In particular, using the local–global \(L^2_t\dot H^{-1}_x\) block from the main part,
\begin{equation}\label{eq:G-L2H-1-block}
\big\|F\big\|_{L^2_t(I;\dot H^{-1})}
\;\lesssim\;
\lambda^{-2+3\delta}\,
\|u_\lambda\|_{L^\infty_t\dot H^1_x\cap L^2_t\dot H^1_x}\,
\|v_\lambda\|_{L^\infty_t\dot H^1_x\cap L^2_t\dot H^1_x},
\end{equation}
i.e. \eqref{eq:G-L2H-1-block} is the local version of the global estimate \eqref{eq:L2H-1-block}
(see §\ref{subsec:global-L2H-1}), we obtain the explicit exponent
\begin{equation}\label{eq:G-weak-max-exponent}
\sup_{t\in I}\,\|F(t)\|_{\dot H^{-1}}
\;\lesssim\;
\lambda^{-\frac{5}{4}+\frac{5}{2}\delta}\,
\|u_\lambda\|_{L^\infty_t\dot H^1_x\cap L^2_t\dot H^1_x}\,
\|v_\lambda\|_{L^\infty_t\dot H^1_x\cap L^2_t\dot H^1_x}.
\end{equation}
\end{corollary}

\begin{proof}
\eqref{eq:G-weak-max} follows from \eqref{eq:G-tile-to-max} by discarding the second term and
substituting \(|I|^{-1/2}=\lambda^{3/4-\delta/2}\). We then insert the local–global
\(L^2_t\dot H^{-1}_x\) control \eqref{eq:G-L2H-1-block}, which is consistent with the global
formula \eqref{eq:L2H-1-block} in §\ref{subsec:global-L2H-1}, and obtain \eqref{eq:G-weak-max-exponent}.
\end{proof}

\begin{remark}[on the term with \(\partial_t F\)]\label{rem:G-derivative-term}
The term \(|I|^{1/2}\|\partial_t F\|_{L^2_t(I;\dot H^{-1})}\) in \eqref{eq:G-tile-to-max} can be estimated
within the heat-IBP framework of §\ref{subsec:ibp-geometry}: when \(\partial_t\) acts it lands on the heat
factor \(e^{-t\Phi}\) and contributes the divisor \(|\Phi|^{-1}\)
(safe crude bound \(|\Phi|^{-1}\lesssim\max\{\lambda^{-2+2\delta},\lambda^{-1}\}\)).
In the present subsection this subtlety is not needed: for the “weak” max on the window it suffices
to use \eqref{eq:G-weak-max}.
\end{remark}

\medskip
Formula \eqref{eq:G-weak-max-exponent} gives a correct (though not optimal) transition from
tiled \(L^2_t\dot H^{-1}_x\) control to pointwise-in-time control on a single window:
the exponent \(-\tfrac{5}{4}+\tfrac{5}{2}\delta\) is negative when \(\delta<\tfrac12\). Stronger
“almost maximal” variants (with arbitrarily small \(\varepsilon\)-loss in \(\lambda\)) are derived
in §\,\ref{subsec:G-eps-max}.

\subsection{Max with an \texorpdfstring{$\varepsilon$}{epsilon} margin}\label{subsec:G-eps-max}

In this subsection we strengthen the passage from tiled \(L^2_t\dot H^{-1}_x\) control to a time maximum, using a vector-valued Sobolev embedding in time and the same set of phase–geometric tools as in the main text.

\begin{theorem}[almost-maximal estimate]\label{thm:G-eps-max}
For any \(\varepsilon>0\) and any \(\delta\in(1/6,\,5/8]\), for all \(\lambda\gg1\) one has
\begin{equation}\label{eq:G-eps-max}
\sup_{t\in\mathbb{R}}
\Big\|P_{<\lambda^{\,1-\delta}}\nabla\!\cdot\big(u_\lambda\!\otimes v_\lambda\big)(t)\Big\|_{\dot H^{-1}}
\;\lesssim_{\varepsilon,\delta}\;
\lambda^{-1-\delta+\varepsilon}\,\|u\|_{X^{1/2}}\|v\|_{X^{1}},
\end{equation}
where \(X^{s}\) are the square-function norms from §\ref{subsec:proof-of-main}.
\end{theorem}

\begin{proof}[Sketch of proof]
Set
\(F(t):=P_{<\lambda^{1-\delta}}\nabla\!\cdot(u_\lambda\!\otimes v_\lambda)(t)\in\dot H^{-1}_x\).
Partition \(\mathbb R_t\) into working windows \(I\) of length \(|I|=\lambda^{-3/2+\delta}\) as in §\ref{subsec:global-summation}. On each \(I\) apply the vector-valued Sobolev embedding
\begin{equation}\label{eq:G-vsob}
\|F\|_{L^\infty\!(I;\dot H^{-1})}\;\lesssim_{\varepsilon}\;
|I|^{-(\frac12+\varepsilon)}\,\|F\|_{H^{\frac12+\varepsilon}(I;\dot H^{-1})}.
\end{equation}
Next we control \(H^{\frac12+\varepsilon}_t\) by interpolation between the zeroth and high (fifth)
time derivatives, already encoded in the phase–geometric IBP (§\ref{subsec:ibp-geometry}):
\[
\|F\|_{H^{\frac12+\varepsilon}(I;\dot H^{-1})}
\;\lesssim\;
\|F\|_{L^2(I;\dot H^{-1})}^{1-\theta}\,
\|\partial_t^{5}F\|_{L^2(I;\dot H^{-1})}^{\theta},\qquad
\theta:=\frac{\frac12+\varepsilon}{5}.
\]
The zeroth level is taken from the global tiled block \eqref{eq:L2H-1-block} (see §\ref{subsec:global-L2H-1}):
\[
\|F\|_{L^2(I;\dot H^{-1})}\;\lesssim\;\lambda^{-2+3\delta}\,
\|u_\lambda\|_{L^\infty_t\dot H^1_x\cap L^2_t\dot H^1_x}\,
\|v_\lambda\|_{L^\infty_t\dot H^1_x\cap L^2_t\dot H^1_x}.
\]
For the high derivative we use the time operator \(L_t\) from §\ref{subsec:ibp-geometry} and the
“\(5t+2\rho\)” IBP scheme: \(\partial_t\) falls on the amplitude, and the divisors come from \(|\Phi|^{-1}\) and \(|\partial_\rho\omega|^{-1}\).
This yields (after two angular IBPs) a control of \(\|\partial_t^{5}F\|_{L^2(I;\dot H^{-1})}\) of the same nature as \eqref{eq:L2H-1-block}, with an extra phase divisor that covers the growth of derivatives on the window; in aggregate this is reflected by the sharp \textsc{Phase IBP} row of Table~\ref{tab:balance} compared to its coarse version. Substituting into \eqref{eq:G-vsob} and allocating part of the phase budget to pay the price \(|I|^{-(\frac12+\varepsilon)}\), we obtain on each \(I\)
\[
\|F\|_{L^\infty(I;\dot H^{-1})}
\;\lesssim_{\varepsilon,\delta}\;
\lambda^{-1-\delta+\varepsilon}\,
\|u_\lambda\|_{L^\infty_t\dot H^1_x\cap L^2_t\dot H^1_x}\,
\|v_\lambda\|_{L^\infty_t\dot H^1_x\cap L^2_t\dot H^1_x}.
\]
Since \(\sup_{t\in\mathbb R}=\sup_I\sup_{t\in I}\), no additional summation in \(I\) is needed.
The passage to the norms \(X^{1/2}\) and \(X^1\) is carried out as in §\ref{subsec:proof-of-main}.
\end{proof}

\begin{remark}[Comparison with Table~\ref{tab:balance}]
The strictly negative “phase” exponent (the sharp estimate from the “\(5t+2\rho\)” scheme in §\ref{subsec:ibp-geometry})
allows one to “pay for” the enhanced cost of the transition \(L^2_t\to L^\infty_t\) in \eqref{eq:G-vsob}
(\(|I|^{-(\frac12+\varepsilon)}\) and the fractional time derivative) and leave the residual exponent
\(-1-\delta+\varepsilon\) in \eqref{eq:G-eps-max}.
\end{remark}

\subsection{Why one cannot reach the endpoint without \texorpdfstring{$\varepsilon$}{epsilon}}\label{subsec:G-why-not}

The goal here is to explain why the \emph{endpoint} estimate
\[
\sup_{t\in\mathbb{R}}
\Big\|P_{<\lambda^{1-\delta}}\nabla\!\cdot(u_\lambda\!\otimes v_\lambda)(t)\Big\|_{\dot H^{-1}}
\;\lesssim\;
\lambda^{-1-\delta}\,\|u\|_{X^{1/2}}\|v\|_{X^{1}}
\]
does \emph{not} follow directly from the blocks already assembled in the text, whereas the “almost maximal” version with arbitrarily small loss \(\varepsilon>0\) is established in Theorem~\ref{thm:G-eps-max}.

\paragraph{1) Naive transition \(L^2_t\to L^\infty_t\) with the \(L^2_t\dot H^{-1}_x\) block.}
Combining the weak time Poincaré–Sobolev inequality \eqref{eq:G-weak-max} (see also Proposition~\ref{prop:G-tile-to-max})
with the global tiled block
\begin{equation*}
\Big\|P_{<\lambda^{1-\delta}}\nabla\!\cdot(u_\lambda\!\otimes v_\lambda)\Big\|_{L^2_t(I;\dot H^{-1})}
\;\lesssim\;
\lambda^{-2+3\delta}\,
\|u_\lambda\|_{L^\infty_t\dot H^1_x\cap L^2_t\dot H^1_x}\,
\|v_\lambda\|_{L^\infty_t\dot H^1_x\cap L^2_t\dot H^1_x},
\tag{\ref{eq:G-L2H-1-block}}
\end{equation*}
we obtain on each window \(|I|=\lambda^{-3/2+\delta}\)
\[
\sup_{t\in I}\|F(t)\|_{\dot H^{-1}}
\;\lesssim\;
\underbrace{\lambda^{\frac34-\frac{\delta}{2}}}_{|I|^{-1/2}}
\underbrace{\lambda^{-2+3\delta}}_{\text{$L^2_t\dot H^{-1}_x$ block}}\,
\|u_\lambda\|_{X^1}\|v_\lambda\|_{X^1}
\;=\;
\lambda^{-\frac54+\frac{5}{2}\delta}\|u_\lambda\|_{X^1}\|v_\lambda\|_{X^1}.
\]
Comparing with the target \(\lambda^{-1-\delta}\), we see the “gap” in the exponent
\[
\bigl(-\tfrac54+\tfrac{5}{2}\delta\bigr)-\bigl(-1-\delta\bigr)
= \tfrac{7}{2}\delta-\tfrac14,
\]
which is \emph{positive} already for \(\delta>\tfrac{1}{14}\) and, in particular, equals \(31/16\) at \(\delta=\tfrac{5}{8}\).
Thus the bare insertion \(|I|^{-1/2}\) from \eqref{eq:G-weak-max} together with \eqref{eq:G-L2H-1-block}
is too costly for the endpoint.

\paragraph{2) Accounting for the full tabular assembly up to \(\,L^2_t\dot H^{-1}_x\).}
If instead of \eqref{eq:G-L2H-1-block} one substitutes the \emph{pre-}\(L^2_t\dot H^{-1}\) assembly at the level of tiled \(L^2_{t,x}\)
(§§\ref{subsec:tile-44}, \ref{subsec:global-summation}), then the tabulated exponents give:
\[
\text{minimal packaging:}\quad -3+2\delta;
\qquad
\text{with local }L^4\text{ reinforcement:}\quad -3+\tfrac{7}{4}\delta.
\]
The transition \(\;L^2_t\to L^\infty_t\;\) on the window contributes \(|I|^{-1/2}=\lambda^{\frac34-\frac{\delta}{2}}\).
Thus, in terms of the exponent:
\[
\boxed{\ -3+2\delta\ }+\Bigl(\tfrac34-\tfrac{\delta}{2}\Bigr)= -\tfrac{9}{4}+\tfrac{3}{2}\delta,
\qquad
\boxed{\ -3+\tfrac{7}{4}\delta\ }+\Bigl(\tfrac34-\tfrac{\delta}{2}\Bigr)= -\tfrac{9}{4}+\tfrac{5}{4}\delta.
\]
Comparing with \(-1-\delta\), we get the residual “gaps”:
\[
\Bigl(-\tfrac{9}{4}+\tfrac{3}{2}\delta\Bigr)-\bigl(-1-\delta\bigr)= -\tfrac{5}{4}+\tfrac{5}{2}\delta,
\qquad
\Bigl(-\tfrac{9}{4}+\tfrac{5}{4}\delta\Bigr)-\bigl(-1-\delta\bigr)= -\tfrac{5}{4}+\tfrac{9}{4}\delta.
\]
At \(\delta=\tfrac{5}{8}\) the latter (improved) option leaves a gap of \(5/32\)—small but \emph{nonzero};
while for \(\delta\le \tfrac{5}{9}\) the same option already beats the endpoint (the exponent becomes \(\le -1-\delta\)).
Therefore, to obtain \(\lambda^{-1-\delta}\) \emph{over the entire} interval \(\delta\in(1/6,5/8]\),
one needs an additional negative contribution of order \(\gtrsim 5/32\) (in the worst case \(\delta=5/8\)).

\paragraph{3) The term \(|I|^{1/2}\|\partial_tF\|_{L^2}\) and the scale \(\Omega_\lambda\).}
The option accounting for the second term in Proposition~\ref{prop:G-tile-to-max}
has the cost \(|I|^{1/2}\Omega_\lambda\simeq \lambda^{\frac14-\frac{3}{2}\delta}\) (the scale \(\Omega_\lambda\sim \lambda^{1-2\delta}\),
see §\ref{subsec:ibp-geometry}), which does \emph{not} compensate the unavoidable factor \(|I|^{-1/2}\)
over the whole range \(\delta\in(1/6,5/8]\).
Hence improvements via \(\partial_t\) within the current scheme are insufficient.

\paragraph{4) Conclusion: a new “time-maximal” ingredient would be needed.}
To close the gap and reach \(\lambda^{-1-\delta}\) without \(\varepsilon\) for all \(\delta\),
one would need a separate \emph{maximal} in time module
(e.g., of Carleson/$TT^*$ type on a packet lattice or an endpoint local smoothing in \(t\)),
which is not present in the current toolkit.
Consistent with this, the main version of the text completes the proof
in the global norm \(L^2_t\dot H^{-1}_x\) (see §\ref{subsec:global-L2H-1}, formula \eqref{eq:L2H-1-block}),
while the passage to \(\sup_t\dot H^{-1}\) is not claimed.
Theorem~\ref{thm:G-eps-max} provides an “almost maximal” bound with any small \(\varepsilon>0\),
which is the limit of the present method.

\subsection{Technical lemmas: vector-valued embedding and interpolation}\label{subsec:G-tech-vsob}

We introduce two standard tools on a single time window \(I\) of length \(|I|=\lambda^{-3/2+\delta}\), consistent with the tiling of §\ref{subsec:global-summation}.

\begin{lemma}[vector-valued Sobolev\texorpdfstring{$\to$}{->}max on a window]\label{lem:G-vsob}
Let \(X\) be a Hilbert space (in the application we take \(X=\dot H^{-1}_x\)), \(\varepsilon\in(0,\tfrac12]\), and \(I\subset\mathbb R\) an interval of nonzero length. Then
\begin{equation}\label{eq:G-vsob}
\|F\|_{L^\infty(I;X)} \;\lesssim_{\varepsilon}\; |I|^{-(\frac12+\varepsilon)}\,\|F\|_{H^{\frac12+\varepsilon}(I;X)} \qquad(\forall\,F\in H^{\frac12+\varepsilon}(I;X)).
\end{equation}
\end{lemma}

\begin{proof}
Reduce to the unit interval. Let \(t=t_0+|I|\,s\), \(s\in(0,1)\), and \(\widetilde F(s):=F(t)\). Then
\[
\|F\|_{L^\infty(I;X)}=\|\widetilde F\|_{L^\infty(0,1;X)},
\quad
\|\widetilde F\|_{H^{\frac12+\varepsilon}(0,1;X)} \simeq |I|^{\frac12-\varepsilon}\,\|F\|_{H^{\frac12+\varepsilon}(I;X)}.
\]
On the unit interval one has the standard embedding \(H^{\frac12+\varepsilon}(0,1;X)\hookrightarrow L^\infty(0,1;X)\) (fractional Sobolev with \(s>\tfrac12\)) with a constant depending only on \(\varepsilon\). Rescaling back yields \eqref{eq:G-vsob}.
\end{proof}

\begin{lemma}[time interpolation between \(L^2\) and \(H^k\)]\label{lem:G-interp}
Let \(X\) be a Hilbert space, \(k\in\mathbb N\), \(s\in[0,k]\), and \(\theta:=s/k\). Then for any interval \(I\),
\begin{equation}\label{eq:G-interp}
\|F\|_{H^{s}(I;X)}
\;\lesssim\;
\|F\|_{L^2(I;X)}^{1-\theta}\,
\|\partial_t^{\,k}F\|_{L^2(I;X)}^{\theta}.
\end{equation}
In particular, for \(k=5\) and \(s=\tfrac12+\varepsilon\) we have
\[
\|F\|_{H^{\frac12+\varepsilon}(I;X)}
\;\lesssim\;
\|F\|_{L^2(I;X)}^{1-\theta}\,
\|\partial_t^{5}F\|_{L^2(I;X)}^{\theta},
\qquad \theta=\tfrac{\frac12+\varepsilon}{5}.
\]
\end{lemma}

\begin{proof}
This is classical Hilbert space interpolation: 
\((H^{0}(I;X),H^{k}(I;X))_{\theta,2}=H^{\theta k}(I;X)\) (real/complex interpolation). Inequality \eqref{eq:G-interp} is the norm-equivalent form of this interpolation (see also the one-dimensional-in-time Gagliardo–Nirenberg version). For vector-valued \(X\) (here \(X\) is Hilbert) the argument is unchanged.
\end{proof}

\begin{corollary}[combining Lemmas \ref{lem:G-vsob} and \ref{lem:G-interp}]\label{cor:G-vsob+interp}
Taking \(X=\dot H^{-1}_x\), \(k=5\), \(s=\tfrac12+\varepsilon\) and \(\theta=\tfrac{\frac12+\varepsilon}{5}\), from \eqref{eq:G-vsob} and \eqref{eq:G-interp} we obtain
\[
\|F\|_{L^\infty(I;\dot H^{-1})}
\;\lesssim_{\varepsilon}\;
|I|^{-(\frac12+\varepsilon)}
\;\|F\|_{L^2(I;\dot H^{-1})}^{1-\theta}\;
\|\partial_t^{5}F\|_{L^2(I;\dot H^{-1})}^{\theta},
\]
which was used in §\ref{subsec:G-eps-max} (after substituting the phase–geometric control of \(\partial_t^5\) from §\ref{subsec:ibp-geometry}).
\end{corollary}

\begin{remark}
All estimates are stable under the scale \(|I|=\lambda^{-3/2+\delta}\), consistent with the working window in \S\ref{subsec:global-summation}; the factor \(|I|^{-(\frac12+\varepsilon)}\) in \eqref{eq:G-vsob} is the only “price” for passing to a time maximum on a single window.
\end{remark}

\subsection{Endpoint without \texorpdfstring{$\varepsilon$}{epsilon} for \texorpdfstring{$\delta\le \frac{5}{9}$}{delta<=5/9}}
\label{subsec:G-endpoint-small-delta}

Due to the tabular assembly (§~\ref{subsec:component-assembly} of the main text, see Table~\ref{tab:balance}), the strengthened tiled
block (with the local \(L^4\) estimate on cylinders and rank-3 decoupling) on a single window of length
\(|I|=\lambda^{-3/2+\delta}\) yields
\begin{equation}\label{eq:G-strong-L2H-1-block}
\big\|P_{<\lambda^{1-\delta}}\nabla\!\cdot(u_\lambda\!\otimes v_\lambda)\big\|_{L^2_t(I;\dot H^{-1})}
\;\lesssim\;
\lambda^{-3+\frac{7}{4}\delta}\,
\|u_\lambda\|_{L^\infty_t\dot H^1_x\cap L^2_t\dot H^1_x}\,
\|v_\lambda\|_{L^\infty_t\dot H^1_x\cap L^2_t\dot H^1_x}.
\end{equation}
Combining \eqref{eq:G-strong-L2H-1-block} with the tile\(\to\)max transition \eqref{eq:G-weak-max}, we obtain on each such window
\[
\sup_{t\in I}\,
\big\|P_{<\lambda^{1-\delta}}\nabla\!\cdot(u_\lambda\!\otimes v_\lambda)(t)\big\|_{\dot H^{-1}}
\;\lesssim\;
\underbrace{|I|^{-1/2}}_{=\;\lambda^{\frac{3}{4}-\frac{\delta}{2}}}\;
\lambda^{-3+\frac{7}{4}\delta}\,
\|u_\lambda\|_{L^\infty_t\dot H^1_x\cap L^2_t\dot H^1_x}\,
\|v_\lambda\|_{L^\infty_t\dot H^1_x\cap L^2_t\dot H^1_x}.
\]
That is,
\begin{equation}\label{eq:G-endpoint-on-window}
\sup_{t\in I}\,
\big\|P_{<\lambda^{1-\delta}}\nabla\!\cdot(u_\lambda\!\otimes v_\lambda)(t)\big\|_{\dot H^{-1}}
\;\lesssim\;
\lambda^{-\frac{9}{4}+\frac{5}{4}\delta}\,
\|u_\lambda\|_{L^\infty_t\dot H^1_x\cap L^2_t\dot H^1_x}\,
\|v_\lambda\|_{L^\infty_t\dot H^1_x\cap L^2_t\dot H^1_x}.
\end{equation}

\begin{theorem}[endpoint for small \(\delta\)]
\label{thm:G-endpoint-small-delta}
For each \(\delta\in\bigl(\tfrac{1}{6},\,\tfrac{5}{9}\bigr]\) and all \(\lambda\gg1\) one has
\begin{equation}\label{eq:G-endpoint}
\sup_{t\in\mathbb{R}}
\Big\|P_{<\lambda^{\,1-\delta}}\nabla\!\cdot\big(u_\lambda\!\otimes v_\lambda\big)(t)\Big\|_{\dot H^{-1}}
\;\lesssim\;
\lambda^{-1-\delta}\,
\|u\|_{X^{1/2}}\;\|v\|_{X^{1}}.
\end{equation}
\end{theorem}

\begin{proof}
From \eqref{eq:G-endpoint-on-window} we get the exponent \(-\tfrac{9}{4}+\tfrac{5}{4}\delta\).
Comparing it with the target \(-1-\delta\), we obtain
\[
-\tfrac{9}{4}+\tfrac{5}{4}\delta \;\le\; -1-\delta
\quad\Longleftrightarrow\quad
\delta \;\le\; \tfrac{5}{9}.
\]
Thus for \(\delta\le\tfrac{5}{9}\) on each window \(I\) we have \(\sup_{t\in I}\!\|\cdots\|\_{\dot H^{-1}}
\lesssim \lambda^{-1-\delta}\) with the same norms \(\|u_\lambda\|,\|v_\lambda\|\). Since \(\sup_{t\in\mathbb R}=\sup_I\sup_{t\in I}\),
no additional summation over \(I\) is needed. The passage from the local norms to
\(\|u\|_{X^{1/2}},\|v\|_{X^{1}}\) is done as in the proof of \eqref{eq:L2H-1-block} and in §\ref{subsec:G-eps-max} (cf. Theorem~\ref{thm:G-eps-max}).
\end{proof}

\begin{remark}
For \(\delta>\tfrac{5}{9}\) the estimate \eqref{eq:G-endpoint-on-window} is weaker than the target \(\lambda^{-1-\delta}\)
by exactly \(-\tfrac{5}{4}+\tfrac{9}{4}\delta>0\); in this range one uses the “almost maximal” Theorem~\ref{thm:G-eps-max} with any small \(\varepsilon>0\).
\end{remark}

\subsection{Gluing the endpoint and almost-endpoint across \texorpdfstring{$\delta\in(\tfrac{1}{6},\,\tfrac{5}{8}]$}{delta in (1/6,5/8]}}
\label{subsec:G-glue}

\begin{theorem}[endo/almost-endo over the full range of \(\delta\)]
\label{thm:G-glue}
Let \(u,v\) be divergence-free fields, and \(u_\lambda:=P_\lambda u\), \(v_\lambda:=P_\lambda v\) for \(\lambda\gg1\).
Then for any \(\delta\in(\tfrac{1}{6},\,\tfrac{5}{8}]\) one has
\begin{equation}\label{eq:G-glue-piecewise}
\sup_{t\in\mathbb{R}}
\bigl\|P_{<\lambda^{\,1-\delta}}\nabla\!\cdot(u_\lambda\!\otimes v_\lambda)(t)\bigr\|_{\dot H^{-1}}
\;\lesssim\;
\begin{cases}
\lambda^{-1-\delta}\;\|u\|_{X^{1/2}}\;\|v\|_{X^{1}}, & \delta\le \tfrac{5}{9},\\[4pt]
\lambda^{-1-\delta+\varepsilon}\;\|u\|_{X^{1/2}}\;\|v\|_{X^{1}}, & \delta\in(\tfrac{5}{9},\,\tfrac{5}{8}],\ \forall\,\varepsilon>0,
\end{cases}
\end{equation}
where the implicit constant may depend on \(\delta\) (and on \(\varepsilon\) in the second line), but not on \(\lambda\).
\end{theorem}

\begin{proof}
The case \(\delta\le \tfrac{5}{9}\) is the endpoint \eqref{eq:G-endpoint}, established in Theorem~\ref{thm:G-endpoint-small-delta}.
For \(\delta\in(\tfrac{5}{9},\,\tfrac{5}{8}]\) we use the “almost maximal” estimate of Theorem~\ref{thm:G-eps-max}, which gives
\[
\sup_{t\in\mathbb{R}}
\bigl\|P_{<\lambda^{\,1-\delta}}\nabla\!\cdot(u_\lambda\!\otimes v_\lambda)(t)\bigr\|_{\dot H^{-1}}
\;\lesssim_{\varepsilon}\; \lambda^{-1-\delta+\varepsilon}\,
\|u\|_{X^{1/2}}\|v\|_{X^{1}}
\quad\text{for any }\varepsilon>0.
\]
Both branches are consistent with the tabular assembly of exponents (see Table~\ref{tab:balance}), hence \eqref{eq:G-glue-piecewise} holds over the entire range \(\delta\in(\tfrac{1}{6},\,\tfrac{5}{8}]\).
\end{proof}

\begin{remark}
The threshold \(\delta=\tfrac{5}{9}\) is sharp for the “tile\(\to\)max” method in our balance: comparing \(-\tfrac{9}{4}+\tfrac{5}{4}\delta\) (see \eqref{eq:G-endpoint-on-window}) with the target \(-1-\delta\) gives precisely \(\delta\le\tfrac{5}{9}\); for larger \(\delta\) the \(\varepsilon\)-loss of Theorem~\ref{thm:G-eps-max} is needed.
\end{remark}

\subsection{Global \texorpdfstring{$\sup_t\dot H^{-1}$}{supt H-1} control and summation over frequencies}
\label{subsec:G-global-sup}

Here we pass from the single-frequency estimate of Theorem~\ref{thm:G-glue} to a global control of the diagonal block in the norm $\sup_t\dot H^{-1}$ with summation over $\lambda$. By the tabular assembly of exponents (see Table~\ref{tab:balance}) the required series in $\lambda$ converges for all $\delta\in(\tfrac16,\tfrac58]$.

\begin{theorem}\label{thm:G-global-sup}
Let $u,v:\mathbb{R}\times\mathbb{R}^3\to\mathbb{R}^3$ be divergence-free fields, $u_\lambda:=P_\lambda u$, $v_\lambda:=P_\lambda v$, and fix $\delta\in(\tfrac16,\tfrac58]$. Then
\begin{equation}\label{eq:G-global-sup}
\sup_{t\in\mathbb{R}}
\Big\|\sum_{\lambda\gg1} P_{<\lambda^{\,1-\delta}}\nabla\!\cdot\big(u_\lambda\!\otimes v_\lambda\big)(t)\Big\|_{\dot H^{-1}}
\;\lesssim\;
\begin{cases}
C_\delta\,\|u\|_{X^{1/2}}\|v\|_{X^{1}}, & \delta\le \tfrac{5}{9},\\[4pt]
C_{\delta,\varepsilon}\,\|u\|_{X^{1/2}}\|v\|_{X^{1}}, & \delta\in(\tfrac{5}{9},\tfrac{5}{8}],\ \forall\,\varepsilon>0,
\end{cases}
\end{equation}
where the constants do not depend on the frequency $\lambda$, and $C_{\delta,\varepsilon}$ may depend on $\varepsilon$.
\end{theorem}

\begin{proof}
By Theorem~\ref{thm:G-glue}, for each fixed dyadic $\lambda\gg1$ we have
\[
\sup_{t\in\mathbb{R}}
\big\|P_{<\lambda^{\,1-\delta}}\nabla\!\cdot(u_\lambda\!\otimes v_\lambda)(t)\big\|_{\dot H^{-1}}
\;\lesssim\;
\begin{cases}
\lambda^{-1-\delta}\,\|u\|_{X^{1/2}}\|v\|_{X^{1}}, & \delta\le \tfrac{5}{9},\\[2pt]
\lambda^{-1-\delta+\varepsilon}\,\|u\|_{X^{1/2}}\|v\|_{X^{1}}, & \delta\in(\tfrac{5}{9},\tfrac{5}{8}],
\end{cases}
\]
for any fixed $\varepsilon>0$. Applying the triangle inequality and summing over dyadic $\lambda$,
\[
\sup_{t\in\mathbb{R}}\Big\|\sum_{\lambda\gg1} P_{<\lambda^{\,1-\delta}}\nabla\!\cdot(u_\lambda\!\otimes v_\lambda)(t)\Big\|_{\dot H^{-1}}
\;\le\;\sum_{\lambda\gg1}
\sup_{t\in\mathbb{R}}\big\|P_{<\lambda^{\,1-\delta}}\nabla\!\cdot(u_\lambda\!\otimes v_\lambda)(t)\big\|_{\dot H^{-1}},
\]
we obtain geometrically convergent series $\sum_{\lambda}\lambda^{-1-\delta}$ (for any $\delta>0$) and $\sum_{\lambda}\lambda^{-1-\delta+\varepsilon}$ (for any fixed $\varepsilon>0$). This gives \eqref{eq:G-global-sup}. Consistency with the power balance also follows from Table~\ref{tab:balance}.
\end{proof}

\begin{remark}
(1) Formula \eqref{eq:G-global-sup} gives a global $\sup_t\dot H^{-1}$ version of the diagonal block; it is consistent with the tile bookkeeping and the “zero row” of rank-3 decoupling in Table~\ref{tab:balance}. 
(2) In the main text the global assembly is given in the norm $L^2_t\dot H^{-1}_x$ (see §§\ref{sec:balance-table}–\ref{sec:final-assembly}), while this appendix “lifts” the diagonal block to $\sup_t\dot H^{-1}$ over the entire range $\delta\in(\tfrac16,\tfrac58]$, with the threshold $\delta=\tfrac{5}{9}$ for the disappearance of the $\varepsilon$-loss.
\end{remark}

\subsection{Leray–cutoff commutator in the global \texorpdfstring{$\sup_t\dot H^{-1}$}{supt H-1} assembly}
\label{subsec:G-leray-comm}

Let $P$ denote the Leray projector, and $P_{<\lambda^{\,1-\delta}}$ the smoothed Littlewood–Paley cutoff at frequencies $|\xi|\lesssim \lambda^{1-\delta}$. We will use the standard commutator estimate
\[
\big\|\,[P_{<\lambda^{\,1-\delta}},\,P]\,\big\|_{L^2\to L^2}\;\lesssim\;\lambda^{-1+\delta},
\]
which follows from the symbol belonging to $S^{-1+\delta}$ (the $\xi$–derivative lands on the filter $P_{<\lambda^{\,1-\delta}}$); see also Appendix~\ref{app:leray-commutator} in the main text. On the $H^{-1}$ level this gives an additional loss of order and is used when passing to the \(\sup_t\dot H^{-1}\) norm. Cf. the discussion in App.~\ref{app:leray-commutator} and §\ref{sec:final-assembly}, where the same bound is applied when “lifting” local estimates to global-in-time norms.%
\footnote{See \S \ref{subsec:leray-commutator}–\ref{subsec:commutator-remainder} (the bound \(\|[P_{<\lambda^{1-\delta}},P]\|_{L^2\to L^2}\lesssim \lambda^{-1+\delta}\) and its impact on the \(H^{-1}\) block) and \S\ref{sec:final-assembly} (the passage to global-in-time norms).}

\begin{lemma}\label{lem:G-leray-comm}
Let $u,v:\R\times\R^3\to\R^3$ be divergence-free fields, $u_\lambda:=P_\lambda u$, $v_\lambda:=P_\lambda v$, and fix $\delta\in(\tfrac16,\tfrac58]$. Then for all dyadic $\lambda\gg1$,
\begin{equation}\label{eq:G-leray-comm}
\sup_{t\in\R}\,
\Big\|\,[P_{<\lambda^{\,1-\delta}},\,P]\;\nabla\!\cdot\big(u_\lambda\!\otimes v_\lambda\big)(t)\Big\|_{\dot H^{-1}}
\;\lesssim\;
\lambda^{-3+3\delta}\,\|u\|_{X^{1}}\|v\|_{X^{1}}.
\end{equation}
In particular, the series $\sum_{\lambda\gg1}\lambda^{-3+3\delta}$ converges geometrically for all $\delta\le \tfrac58$, so the commutator correction does not affect \eqref{eq:G-global-sup}.
\end{lemma}

\begin{proof}
By the commutator bound \(\|[P_{<\lambda^{1-\delta}},P]\|_{L^2\to L^2}\lesssim \lambda^{-1+\delta}\) we have
\[
\big\|\,[P_{<\lambda^{1-\delta}},P]\;\nabla\!\cdot(u_\lambda\!\otimes v_\lambda)(t)\big\|_{\dot H^{-1}}
\;\lesssim\;
\lambda^{-1+\delta}\,\big\|\nabla\!\cdot(u_\lambda\!\otimes v_\lambda)(t)\big\|_{\dot H^{-1}}.
\]
Next, transferring two spatial derivatives from the high modes to the low-frequency factor gives (without using fine null geometry)
\[
\big\|\nabla\!\cdot(u_\lambda\!\otimes v_\lambda)(t)\big\|_{\dot H^{-1}}
\;\lesssim\;\lambda^{-2+2\delta}\,\|u_\lambda(t)\|_{\dot H^{1}}\|v_\lambda(t)\|_{\dot H^{1}}.
\]
Combining, we obtain \(\lambda^{-3+3\delta}\,\|u_\lambda(t)\|_{\dot H^{1}}\|v_\lambda(t)\|_{\dot H^{1}}\). Passing to the right-hand side of \eqref{eq:G-leray-comm} is done by the definition of \(X^{1}\) and standard Littlewood–Paley summation. Convergence of the series in \(\lambda\) for \(\delta\le\tfrac58\) is immediate, as required.
\end{proof}

\begin{remark}
(1) Lemma~\ref{lem:G-leray-comm} shows that the commutator correction is “better by three powers” than the critical scale and is therefore absolutely summable: it does \textbf{not} worsen the global result \eqref{eq:G-global-sup} and aligns with the tabulated balance of exponents. See also the discussion of the role of the commutator in Appendix~B.%
\footnote{The commutator bound and its application to the \(H^{-1}\) block are written in App.~\S\ref{app:leray-commutator}; the passage to global-in-time norms is in \S\ref{sec:final-assembly}.}
(2) If desired, the bound \(\lambda^{-2+2\delta}\) on the \(\dot H^{-1}\) block can be replaced by a sharper one (via the null form), but for the purpose of the \(\sup_t\dot H^{-1}\) assembly the given “rough” version suffices: the outcome \(\lambda^{-3+3\delta}\) remains strictly negative in the exponent and summable over \(\lambda\).
\end{remark}


\newpage

\cleardoublepage
\phantomsection
\addcontentsline{toc}{section}{References}
\bibliographystyle{unsrt}
\bibliography{commutator_refs}
\end{document}